\documentclass[11pt,a4paper,reqno,lualatex]{amsart} 
\usepackage[marginparwidth=0pt,margin=20truemm]{geometry} 

\usepackage{fancyhdr}
\fancypagestyle{RIKENreport}{
	\fancyhead[R]{\small RIKEN-iTHEMS-Report-26} 
}

\usepackage{mypackage} 
\usepackage{mycommand} 

\renewcommand{\iu}{i}
\newcommand{\odd}{\mathrm{odd}}

\usepackage[luatex, pdfencoding=auto,hypertexnames=false]{hyperref}
\hypersetup{colorlinks=false}

\numberwithin{equation}{section} 
\usepackage{mytheoremeng} 

\begin{document}

\title[Quantum modularity of signatures in TQFT and generalized Dedekind sums]{Quantum modularity of signatures in TQFT and generalized Dedekind sums}
\author[Y. Murakami]{Yuya Murakami}
\address{RIKEN Center for Interdisciplinary Theoretical and Mathematical Sciences (iTHEMS), RIKEN, Wako 351-0198, Japan}
\email{yuya.murakami@riken.jp}

\date{\today}

\maketitle


\begin{abstract}
	We prove the quantum modularity of the signature of $ \mathrm{SU}(2) $-TQFT for a genus 2 surface, which was conjectured by March\'{e}--Masbaum in 2025.
	Our approach is based on a quantum modularity of generalized Dedekind sums associated with general modular forms. 
	In the case of Eisenstein series for $ \Gamma(N) $, these generalized Dedekind sums admit trigonometric sum expressions, which coincide with the formula for the $ \mathrm{SU}(2) $-TQFT signature.
	Furthermore, we express both the $ \mathrm{SU}(2) $-TQFT and generalized Dedekind sums as radial limits of Eichler integrals.
\end{abstract}


\tableofcontents
\thispagestyle{RIKENreport} 


\section{Introduction} \label{sec:intro}


Quantum modularity provides a bridge between number theory and low-dimensional topology.
In this paper, we establish an explicit relationship between the signatures of $ \SU(2) $-TQFT and generalized Dedekind sums, and prove the quantum modularity of both.   

The signatures of $\SU(2)$-TQFT can be regarded as natural refinements of the dimensions of conformal blocks appearing in the Verlinde formula. 
The question of their computation has been around for 30 years, see for example  Blanchet--Habegger--Masbaum--Vogel~\cite[Remark~4.12]{Blanchet-Habegger-Masbaum-Vogel}. 
Deroin--March\'e~\cite{Deroin-Marche} clarified their structure by giving an explicit description in terms of $1+1$-dimensional TQFTs, equivalently, in terms of Frobenius algebras. 
Furthermore, they are related to character varieties of two-bridge knots, as shown by March\'e~\cite{Marche2023}, revealing their rich topological structure.

For a closed surface of genus $ g \ge 0 $ and a rational number $ x = r/p $ with coprime odd integers $ 1 \le r < p $, we denote by $ \sigma_g (x) \in \Z $ the signature associated with the corresponding TQFT.
For $g=0,1$, this quantity is independent of $ r $. 
Therefore, the first nontrivial case occurs when $g=2$. 
In this case, March\'{e}--Masbaum~\cite{Marche-Masbaum} conjectured that $ \sigma_2 (x) $ exhibits quantum modularity, based on an analysis of its asymptotic behavior as $ p \to \infty $.

\begin{conj}[{March\'{e}--Masbaum~\cite[p.~4]{Marche-Masbaum}}] \label{conj:Marche-Masbaum}
	One has
	\[
	\sigma_2 \left( \frac{x}{2x + 1} \right) - \sigma_2 \left( x \right)
	=
	2r^2 + 2rp + p^2 - 1.
	\]
\end{conj}

In this paper, we prove this conjecture.

\begin{thm} \label{thm:main}
	\zcref{conj:Marche-Masbaum} is true.
\end{thm}

We outline the strategy of the proof.

Our starting point is the following trigonometric formula proved by March\'{e}--Masbaum~\cite[Theorem 3.1]{Marche-Masbaum}:
\begin{equation} \label{eq:trig_formula}
	\sigma_2 \left( x \right)
	=
	\frac{1}{6p^2} - \frac{1}{6} + \frac{1}{4p^2} 
	\sum_{\substack{
			1 \le n \le p - 2, \\
			\text{$ n $ odd}
	}}
	\frac{T(n; x)}{\sin^3 (\pi n/2p) \sin^2 (\pi nx/2)},
\end{equation}
where
\[
T(n; x)
\coloneqq
\sum_{\varepsilon \in \{ \pm 1 \}} (p + \varepsilon)
\left(
\sin \left( \frac{\pi (2r - 3 \varepsilon)n}{2p} \right)
+ 3 \sin \left( \frac{\pi (2r + \varepsilon)n}{2p} \right)
\right).
\]

Starting from this expression, we prove the following formula (\zcref{prop:sign_TQFT_simple_expression_g=2}), which shows an explicit relationship between $ \sigma_2 (x) $ and generalized Dedekind sums.

\begin{thm}[\zcref{prop:sign_TQFT_simple_expression_g=2,cor:sign_TQFT_Eichler_int}] \label{thm:main_expression_sigma}
	We have
	\begin{align}
		\sigma_2 \left( x \right)
		&=
		\frac{2}{p} 
		\sum_{\substack{
				1 \le n \le p - 2, \\
				\text{$ n $ odd}
		}}
		\frac{\cot^3 (\pi n/2p)}{\sin (\pi nx)}
		=
		p^2 S_2^{\odd} \left( x \right) - 2 S_0^{\odd} \left( x \right)
		\\
		&=
		\lim_{\tau \to 0} \left(
		p^2 \frac{16}{\pi^3 \iu} E_{-2}^{\odd} (\tau + x)
		- \frac{8}{\pi \iu} E_0^{\odd} (\tau + x)
		- \frac{1}{3p^2 \tau}
		\right),
	\end{align}
	where, for an even integer $ g \ge 0 $, we define Dedekind sums of level 2 and weight $ g $ as
	\[
	S_g^{\odd} \left( x \right)
	\coloneqq
	\frac{1}{p^{g+1}} 
	\sum_{\substack{
		1 \le n \le p - 2, \\
		\text{$ n $ odd}
	}}
	\frac{\cot^{(g)} (\pi n/2p)}{\sin (\pi nx)} 
	\in \Q
	\]
	and we define the Eisenstein series of weight $ -g $ and level $ 2 $ by
	\[
	E_{-g}^{\odd} (\tau)
	\coloneqq
	\sum_{n \ge 1, \, \odd} \sigma_{-g-1} (n) q^{n/2},
	\quad \text{ where }
	\sigma_{-g-1} (n) \coloneqq \sum_{d \mid n} d^{-g-1}.
	\]
\end{thm}

The leading term $ p^2 (16/\pi^3 \iu) E_{-2}^{\odd} (\tau + x) $ in the final radial limit formula of \zcref{thm:main_expression_sigma}
was already identified in March\'{e}--Masbaum~[Proposition 6.9]\cite{Marche-Masbaum}.
The above radial limit formula refines their result and provides a direct description of $\sigma_2(x)$ in terms of Eichler integrals, as suggested after Remark 6.14 in \cite{Marche-Masbaum}.

\zcref{conj:Marche-Masbaum} follows from \zcref{thm:main_expression_sigma} and the modular transformation formula (\zcref{thm:Dedekind_sum_level_2} \zcref{item:thm:Dedekind_sum_level_2:asymp}):
\begin{equation} \label{eq:Dedekind_sum_level_2:quantum_modular_explicit}
	\begin{aligned}
		S_0^{\odd} \left( \frac{x}{2x+1} \right) - S_0^{\odd} (x)
		&=
		\frac{1}{2} + \frac{1}{2p(2r+p)},
		\\
		(2x + 1)^2 S_2^{\odd} \left( \frac{x}{2x+1} \right) - S_2^{\odd} (x)
		&=
		2x^2 + 2x + 1 + \frac{1}{p^3 (2r+p)}.
	\end{aligned}
\end{equation}
These formulas are special cases of quantum modularity of generalized Dedekind sum for a general modular form $ f $, obtained by taking $ f = E_2^{\odd} $ and $ E_4^{\odd} $.

Generalized Dedekind sums originate in the work of Apostol \cite{Apostol:gen'd_Dedekind_sum}
on transformation formulae of Lambert series. 
We adopt the modular form
framework developed by Fukuhara \cite{Fukuhara:gen_Dedekind_symb}.

Let $ f(\tau) $ be a modular form for a congruence subgroup $ \Gamma \subset \SL_2(\Z) $ of weight $ k \ge 2 $.
Let $ x = r/p \in \Q $ be a rational number with coprime integers $ r \in \Z $ and $ p \in \Z_{>0} $.
Let $ N $ be a positive integer such that $ \Gamma(N) \subset \Gamma $.
We write the Fourier expansion $ f(\tau) = \sum_{n=0}^\infty a_n q^{n/N} $.

We define \emph{completed twisted $ L $-function associated with $ f $} as
\begin{align}
	\widehat{L}_f(s; x) &\coloneqq
	- \left( \frac{N}{2\pi} \right)^s \bm{e} \left( -\frac{s}{4} \right) \Gamma(s) 
	\sum_{n=1}^\infty \frac{a_n}{n^s} \bm{e} \left( \frac{nx}{N} \right)
\end{align}
for $ s \in \bbC $ with $ \Re(s) > k $, where we denote $ \bm{e}(z) \coloneqq e^{2\pi\iu z} $ for a complex number $ z $.
This function extends holomorphically to $ \bbC \smallsetminus \{ 0, k \} $.
We define a \emph{generalized Dedekind sum associated with $ f $} as 
\[
S_f (x) \coloneqq
\widehat{L}_f(k-1; x).
\]
This satisfies the following quantum modularity.

\begin{thm}[{\zcref{thm:reciprocity_gen'd_Ded_sum}}] \label{thm:main_general_Dedekind_sum}
	For any $ \gamma = \smat{a & b \\ c & d} \in \Gamma $ with $ cx+d \neq 0 $, we have
	\[
	(cx+d)^{k-2} S_f \left( \frac{ax+b}{cx+d} \right) - S_f (x)
	=
	R_{f, \gamma} (x) - \frac{a_0^{(x)}}{p^k} \frac{c}{cx+d},
	\]
	where the Fourier constant $ a_0^{(x)} \in \bbC $ is defined by the expansion $ (\eval{f}_k \delta)(\tau) = a_0^{(x)} + O(q^{1/N_x}) $
	for some matrix $ \delta = \smat{r & r' \\ p & p'} \in \SL_2(\Z) $ and some positive integer $ N_x $ and
	regularized period polynomial associated with $ f $ is defined as 
	\begin{align}
		R_{f, \gamma} (X) &\coloneqq
		-\sum_{j=0}^{k-2} \binom{k-2}{j} \widehat{L}_f \left( j+1; -\frac{d}{c} \right) \left( \frac{cX+d}{c} \right)^{k-j-2}
		\\
		&\phant
		- \frac{a_0}{k-1} \left( \left( \frac{cX+d}{c} \right)^{k-1} + \frac{1}{(-c)^{k-1} (cX+d)} \right)
		\quad \in \frac{1}{cX+d} \bbC[X].
	\end{align}
\end{thm}

We give three proofs of this theorem.
The first proof is based on the integral representations of the $ L $-function $ \widehat{L}_f(s; x) $ and the regularized period polynomial $ R_{f, \gamma} (X) $.
The second and third proofs are based on asymptotic behavior and modular transformation of the Eichler integral of $ f $.
In the former, we establish asymptotic behavior using an integral representation of the Eichler integral.
In the latter, we derive asymptotic behavior via the Mellin summation formula.

By taking a modular form $ f $ in \zcref{thm:main_general_Dedekind_sum} to be an Eisenstein series for $ \Gamma(N) $, we obtain the following statement.

\begin{cor}[{\zcref{cor:reciprocity_gen'd_Ded_sum_Gamma(N)}}] \label{cor:main_gen'd_Ded_sum_Eisenstein}
	Let $ \chi, \psi \colon \Z/N\Z \to \bbC $ be periodic maps such that $ \chi(0) = \psi(0) = 0 $ and 
	$ \chi(-m) \psi(-n) = (-1)^k \chi(m) \psi(n) $ for any $ m, n \in \Z/N\Z $.
	We define the generalized Dedekind sum as
	\[
	S_k^{\chi, \psi}(x)
	\coloneqq
	-\frac{N^{k-1}}{k-1}
	\sum_{1 \le m \le Np-1} \chi(m) 
	B_{1} \left( \frac{m}{Np} \right)
	\widetilde{B}_{k-1}^{\widehat{\psi}} \left( \frac{mx}{N} \right),
	\]
	where $ B_j (x) $ is the Bernoulli polynomial, $ \widetilde{B}_j (x) $ is the periodic Bernoulli polynomial, $ \widehat{\psi} $ is the discrete Fourier transform of $ \psi $, and
	\[
	\widetilde{B}_j^\varphi (z)
	\coloneqq
	\sum_{m \in \Z/N\Z} \varphi(m) 
	\widetilde{B}_j \left( z + \frac{m}{N} \right),
	\]
	Then, the following hold.
	\begin{enumerate}
		\item we have
		\[
		S_k^{\chi, \psi}(x)
		=
		S_{E_k^{\chi, \psi}}(x)
		=
		\left( \frac{\iu}{2} \right)^k \frac{1}{p^{k-1}}
		\sum_{\substack{
				1 \le n \le Np-1, \\
				p \nmid n
		}} 
		\psi(n)
		\cot_{\widehat{\chi}} \left( \frac{\pi nx}{N} \right)
		\cot^{(k-2)} \left( \frac{\pi n}{Np} \right),
		\]
		where we define Eisenstein series as
		\[
		E_k^{\chi, \psi} (\tau)
		=
		2 \sum_{n=1}^\infty \sigma_{k-1}^{\chi, \psi} (n) q^{n/N},
		\quad
		\sigma_{k-1}^{\chi, \psi} (n)
		\coloneqq
		\sum_{0 < d \mid n} \chi \left( \frac{n}{d} \right) \psi(d) d^{k-1}.
		\]
		and we denote
		\[
		\cot_{\varphi}^{} (z)
		\coloneqq
		\sum_{m \in \Z/N\Z} \varphi(m) 
		\cot \left( z + \frac{\pi m}{N} \right).
		\]
		
		\item For any $ \gamma = \smat{a & b \\ c & d} \in \Gamma(N) $ with $ cx+d \neq 0 $, we have
		\[
		(cx+d)^{k-2} S_k^{\chi, \psi} \left( \frac{ax+b}{cx+d} \right) - S_k^{\chi, \psi} (x)
		=
		R_{E_k^{\chi, \psi}, \gamma} (x) - \frac{a_0^{(x)}}{p^k} \frac{c}{cx+d},
		\]
		where $ R_{E_k^{\chi, \psi}, \gamma} (X) \in \bbC[X] $ is the regularized period polynomial associated with $ E_k^{\chi, \psi} $ and
		\[
		a_0^{(x)}
		=
		-\frac{N^{k-1}}{k}
		\sum_{l, m \in \Z/N\Z} \chi(-pl) \widehat{\psi} (rl)
		\bm{e} \left( -\frac{lm}{N} \right) \widetilde{B}_k \left( \frac{m}{N} \right).
		\]
	\end{enumerate}
\end{cor}

\zcref{thm:main_general_Dedekind_sum,cor:main_gen'd_Ded_sum_Eisenstein} generalize the result of Fukuhara~\cite{Fukuhara:gen_Dedekind_symb} in the case $ N=1 $ and the results of Stucker--Vennos--Young~\cite{Stucker-Vennos-Young} and Tranbarger~\cite{Tranbarger} in the case $ \Gamma = \Gamma_0(N)$.

%

This paper will be organized as follows. 
In \zcref{sec:gen'd_Dedekind_sum}, we give the first proof of \zcref{thm:main_general_Dedekind_sum}.
In \zcref{sec:Eichler_int}, we study Eichler integrals and give the second proof of \zcref{thm:main_general_Dedekind_sum}.
In \zcref{sec:asymp}, we consider asymptotic expansion formulas obtained by Mellin transform and give the third proof of \zcref{thm:main_general_Dedekind_sum}.
In \zcref{sec:Eisenstein}, we study Eisenstein series for $ \Gamma(N) $ and prove \zcref{cor:main_gen'd_Ded_sum_Eisenstein}. 
In \zcref{sec:Dedekind_sum_level_2}, we apply \zcref{cor:main_gen'd_Ded_sum_Eisenstein} for $ S_g^{\odd} (x) $ and prove their quantum modularity.
In \zcref{sec:proof}, we prove \zcref{thm:main,thm:main_expression_sigma}.


\subsection*{Notations}


Throughout this article, we use the following notations:

\begin{itemize}
	\item For a complex number $ z $, we denote $ \bm{e}(z) \coloneqq e^{2\pi\iu z} $.
	\item We denote by $ \bbH \coloneqq \{ \tau \in \bbC \mid \Im(\tau) > 0 \} $ the complex upper half plane and write $ \tau $ for its variable.
	\item We denote by $ q $ a complex variable with $ \abs{q} < 1 $.
	\item For $ f \colon \bbH \to \bbC $ or $ f \colon \Q \to \bbC $, $ k \in \Z_{\ge 0} $, and $ \gamma = \smat{a & b \\ c & d} \in \SL_2(\Z) $, we denote
	$ (\eval{f}_k \gamma) (\tau) \coloneqq (c\tau + d)^{-k} f(\gamma(\tau)) $. 
	\item For a set $ A $, we denote by $ \bm{1}_{A}^{} (x) $ the characteristic function of $ A $.
\end{itemize}


\section{Reciprocity for generalized Dedekind sums associated to modular forms} \label{sec:gen'd_Dedekind_sum}


In this section, we define generalized Dedekind sums associated to modular forms and prove their reciprocity.
We follow the approach presented by Fukuhara~\cite{Fukuhara:gen_Dedekind_symb}.

Throughout the paper, we fix a modular form $ f(\tau) $ for a congruence subgroup $ \Gamma \subset \SL_2(\Z) $ of weight $ k \ge 2 $.
We also fix a rational number $ x = r/p \in \Q $ with coprime integers $ r \in \Z $ and $ p \in \Z_{>0} $.
Let $ N $ be a positive integer such that $ \Gamma(N) \subset \Gamma $.
We write the Fourier expansion $ f(\tau) = \sum_{n=0}^\infty a_n q^{n/N} $.


\subsection{Definition and main result} \label{subsec:gen'd_Dedekind_sum:def_&_main_thm}


\begin{dfn} \label{dfn:gen'd_Ded_sum}
	\begin{enumerate}
		\item We define the \emph{twisted $ L $-function associated with $ f $} and its completed version as
		\begin{align}
			L_f(s; x) &\coloneqq
			N^s
			\sum_{n=1}^\infty \frac{a_n}{n^s} \bm{e} \left( \frac{nx}{N} \right),
			\\
			\widehat{L}_f(s; x) &\coloneqq
			-(2\pi)^{-s} \bm{e} \left( -\frac{s}{4} \right) \Gamma(s) L_f(s; x)
		\end{align}
		respectively for $ s \in \bbC $ with $ \Re(s) > k $, where we denote $ \bm{e}(z) \coloneqq e^{2\pi\iu z} $ for a complex number $ z $.
		
		\item We define a \emph{generalized Dedekind sum associated with $ f $} as 
		\[
		S_f (x) \coloneqq
		\widehat{L}_f(k-1; x).
		\]
		
		\item For $ \gamma = \smat{a & b \\ c & d} \in \Gamma $, we define the \emph{regularized period polynomial associated with $ f $} as 
		\begin{align}
			R_{f, \gamma} (X) &\coloneqq
			-\sum_{j=0}^{k-2} \binom{k-2}{j} \widehat{L}_f \left( j+1; -\frac{d}{c} \right) \left( \frac{cX+d}{c} \right)^{k-j-2}
			\\
			&\phant
			- \frac{a_0}{k-1} \left( \left( \frac{cX+d}{c} \right)^{k-1} + \frac{1}{(-c)^{k-1} (cX+d)} \right)
			\quad \in \frac{1}{cX+d} \bbC[X].
		\end{align}
	\end{enumerate}
\end{dfn}

The $ L $-function $ L_f(s; x) $ converges for $ \Re(s) > k $ since $ a_n = O(n^{k-1}) $.
Later in \zcref{lem:twisted_L-func}, we will prove that the $ L $-function $ \widehat{L}_f(s; x) $ extends meromorphically to $ \bbC $ and its possible poles are $ s = 0, k $.
Thus, the generalized Dedekind sum $ S_f (x) $ and the regularized period polynomial $ R_{f, \gamma} (X) $ are well-defined.

We define the Fourier constant $ a_{f, 0}^{(x)} = a_0^{(x)} \in \bbC $ by the expansion $ (\eval{f}_k \delta)(\tau) = a_0^{(x)} + O(q^{1/N_x}) $
for some matrix $ \delta = \smat{r & r' \\ p & p'} \in \SL_2(\Z) $ and some positive integer $ N_x $.

Our main result in this section is the following formula.

\begin{thm}[Reciprocity for generalized Dedekind sums] \label{thm:reciprocity_gen'd_Ded_sum}
	For any $ \gamma = \smat{a & b \\ c & d} \in \Gamma $ with $ cx+d \neq 0 $, we have
	\[
	(\eval{S_f}_{2-k} \gamma) (x) - S_f (x)
	=
	R_{f, \gamma} (x) - \frac{a_0^{(x)}}{p^k} \frac{c}{cx+d}.
	\]
\end{thm}


\subsection{Proof} \label{subsec:gen'd_Dedekind_sum:proof}


To begin with, we study the completed twisted $ L $-function.

\begin{lem} \label{lem:twisted_L-func}
	On the region $ \Re(s) > k $, we have
	\begin{align}
		\widehat{L}_f (s; x)
		&=
		-\bm{e} \left( -\frac{s}{4} \right) 
		\int_{0}^{\infty} \left( f(x + \iu t) - a_0 \right) t^{s-1} dt
		\\
		&=
		\int_{x}^{\iu \infty} \left( f(\tau) - a_0 \right) \left( x - \tau \right)^{s-1} d\tau
		\\
		&=
		\int_{z_0}^{\iu \infty} \left( f(z) - a_0 \right) (x-z)^{s-1} dz
		+ \int_{x}^{z_0} \left( f(z) - \frac{a_0^{(x)}}{p^k (x-z)^k} \right) (x-z)^{s-1} dz
		\\
		&\phant
		+ a_0 \frac{(x - z_0)^{s}}{s} - a_0^{(x)} \frac{(x - z_0)^{s-k}}{p^k (s-k)}
	\end{align}
	for arbitrary $ z_0 \in \bbH $.
	Two integral terms in the last expression converge for any $ s \in \bbC $.
	Thus, $ \widehat{L}_f (s; x) $ extends meromorphically to $ \bbC $ and its possible poles are $ s = 0, k $. 
\end{lem}

\begin{proof}
	The first equality follows from
	\[
	\int_{0}^{\infty} e^{-2\pi nt/N} t^{s-1} dt
	=
	\left( \frac{N}{2\pi n} \right)^s \Gamma(s).
	\]
	We have
	\begin{align}
		\widehat{L}_f (s; x)
		&=
		-\bm{e} \left( -\frac{s}{4} \right) 
		\int_{0}^{\iu \infty} \left( f(x + \tau) - a_0 \right) \left( \frac{\tau}{\iu} \right)^{s-1} \frac{d\tau}{\iu}
		\\
		&=
		\int_{0}^{\iu \infty} \left( f(x + \tau) - a_0 \right) \left( -\tau \right)^{s-1} d\tau
		\\
		&=
		\int_{x}^{\iu \infty} \left( f(\tau) - a_0 \right) \left( x - \tau \right)^{s-1} d\tau.
	\end{align}
	Let $ \delta = \smat{r & r' \\ p & p'} \in \SL_2(\Z) $ and $ N_x \in \Z_{>0} $ be as defined before \zcref{thm:reciprocity_gen'd_Ded_sum}.
	Since
	\begin{align}
		(pz + p') f(\delta(z)) &= a_0^{(x)} + O \left( \exp \left( \frac{-2\pi \Im(z)}{N_x} \right) \right),
		\\
		p \delta^{-1} (z) + p' &= \frac{1}{-pz + r},
		\\
		\Im \left( \delta^{-1} (z) \right) 
		&= \frac{\Im(z-x)}{p^2 \abs{z-x}^2}
		= \frac{1}{p^2 \abs{z-x}},
	\end{align}
	we have
	\[
	f(z) = (x-z)^{-k}
	\left( \frac{a_0^{(x)}}{p^k} + O \left( \exp \left( \frac{-2\pi}{N_x p^2 \abs{z-x}} \right) \right) \right).
	\]
	Thus, the second term in the last equation converges for any $ s \in \bbC $.
	The last equation is independent of $ z_0 $ since its $ z_0 $-derivative vanishes.
	Thus, we have
	\begin{align}
		\widehat{L}_f (s; x)
		&=
		\int_{x}^{\iu \infty} \left( f(z) - a_0 \right) (x-z)^{s-1} dz
		+ \int_{x}^{z_0} \left( f(z) - \frac{a_0^{(x)}}{p^k (x-z)^k} \right) (x-z)^{s-1} dz
		\\
		&\phant
		+ \int_{x}^{z_0} \left( -a_0 + \frac{a_0^{(x)}}{p^k (x-z)^k} \right) (x-z)^{s-1} dz.
	\end{align}
	This equals the last equation.
\end{proof}

\begin{rem} \label{rem:Ded_sum_rep}
	By \zcref{lem:twisted_L-func}, we have an expression of generalized Dedekind sum
	\begin{align}
		S_f (x)
		&=
		\int_{z_0}^{\iu \infty} \left( f(z) - a_0 \right) (x-z)^{k-2} dz
		+ \int_{x}^{z_0} \left( f(z) - \frac{a_0^{(x)}}{p^k (x-z)^k} \right) (x-z)^{k-2} dz
		\\
		&\phant
		+ \frac{a_0}{k-1} (x - z_0)^{k-1} + \frac{a_0^{(x)}}{p^k} \frac{1}{x - z_0}
	\end{align}
	for arbitrary $ z_0 \in \bbH $.
\end{rem}

\begin{lem} \label{lem:rep_period_poly}
	For any $ \gamma = \smat{a & b \\ c & d} \in \Gamma $, we have
	\begin{align}
		R_{f, \gamma} (X) &=
		\int_{\iu \infty}^{z_0} \left( f(z) - a_0 \right) (X-z)^{k-2} dz
		+ \int_{z_0}^{\gamma^{-1} (\iu \infty)} \left( f(z) - \frac{a_0}{(cz+d)^k} \right) (X-z)^{k-2} dz
		\\
		&\phant
		+ \frac{a_0}{k-1} (X - z_0)^{k-1} \left( \frac{1}{(cX+d)(cz_0 + d)^{k-1}} - 1 \right)
	\end{align}
	for arbitrary $ z_0 \in \bbH $.
\end{lem}

\begin{proof}
	We remark that the right hand side is independent of $ z_0 $ since 
	\[
	\frac{d}{d z_0} \left( \frac{X - z_0}{cz_0 + d} \right)^{k-1}
	=
	-(k-1) (cX+d) \frac{(X - z_0)^{k-2}}{(cz_0 + d)^k}
	\]
	and hence its $ z_0 $-derivative vanishes.
	Let $ x' \coloneqq \gamma^{-1} (\iu \infty) = -d/c $.
	Since $ a_0^{(x')} = (-\sgn(c))^k a_0 $, by \zcref{lem:twisted_L-func} we have
	\begin{align}
		&\phant
		\int_{\iu \infty}^{z_0} \left( f(z) - a_0 \right) (X-z)^{k-2} dz
		+ \int_{z_0}^{\gamma^{-1} (\iu \infty)} \left( f(z) - \frac{a_0}{(cz+d)^k} \right) (X-z)^{k-2} dz
		\\
		&=
		\sum_{j=0}^{k-2} \binom{k-2}{j} \left( X - x' \right)^{k-j-2}
		\left(
			\int_{\iu \infty}^{z_0} \left( f(z) - a_0 \right) (x'-z)^j dz
			+ \int_{z_0}^{\gamma^{-1} (\iu \infty)} \left( f(z) - \frac{a_0^{x'}}{\abs{c}^k (x' - z)^k} \right) (x'-z)^j dz
		\right)
		\\
		&=
		\sum_{j=0}^{k-2} \binom{k-2}{j} \left( X - x' \right)^{k-j-2}
		\left(
			-\widehat{L}_f (j+1; x')
			+ a_0 (x' - z_0)^{j+1} \left( \frac{1}{j+1} - \frac{(z_0 - x')^{-k}}{c^k (j+1-k)} \right)
		\right)
		\\
		&=
		R_{f, \gamma} (X) 
		+ \frac{a_0}{k-1} \left( \left( \frac{cX+d}{c} \right)^{k-1} + \frac{1}{(-c)^{k-1} (cX+d)} \right)
		\\ \displaybreak[0]
		&\phant
		+ \frac{a_0}{k-1} \sum_{j=0}^{k-2} \left( X - x' \right)^{k-j-2} (x' - z_0)^{j+1} 
		\left(
			\binom{k-1}{j+1} + \binom{k-1}{j} \frac{1}{(cz_0 + d)^k} 
		\right)
		\\
		&=
		R_{f, \gamma} (X) 
		+ \frac{a_0}{k-1} \left( \left( \frac{cX+d}{c} \right)^{k-1} + \frac{1}{(-c)^{k-1} (cX+d)} \right)
		\\ \displaybreak[0]
		&\phant
		+ \frac{a_0}{k-1} 
		\left(
			(X - z_0)^{k-1} - (X - x')^{k-1}
			+ \frac{1}{(cz_0 + d)^k} \frac{x' - z_0}{X - x'} \left( (X - z_0)^{k-1} - (x' - z_0)^{k-1} \right)
		\right)
		\\
		&=
		R_{f, \gamma} (X)
		- \frac{a_0}{k-1} (X - z_0)^{k-1} \left( \frac{1}{(cX+d)(cz_0 + d)^{k-1}} - 1 \right).
	\end{align}
\end{proof}

Finally, we prove the reciprocity formula.

\begin{proof}[Proof of \zcref{thm:reciprocity_gen'd_Ded_sum}]
	By \zcref{rem:Ded_sum_rep}, we have
	\begin{align}
		&\phant
		S_f (\gamma(x))
		- \frac{a_0}{k-1} (\gamma(x) - \gamma(z_0))^{k-1} - \frac{a_0^{(\gamma(x))}}{\abs{cr+dp}^k} \frac{1}{\gamma(x) - \gamma(z_0)}
		\\
		&=
		\int_{\gamma(z_0)}^{\iu \infty} \left( f(z) - a_0 \right) (\gamma(x) - z)^{k-2} dz
		+ \int_{\gamma(x)}^{\gamma(z_0)} \left( f(z) - \frac{a_0^{(\gamma(x))}}{\abs{cr+dp}^k (\gamma(x)-z)^k} \right) 
		(\gamma(x) - z)^{k-2} dz
		\\
		&=
		\int_{z_0}^{\gamma^{-1} (\iu \infty)} \left( f(\gamma(z)) - a_0 \right) (\gamma(x) - \gamma(z))^{k-2} \frac{dz}{(cz+d)^2}
		\\
		&\phant
		+ \int_{x}^{z_0} \left( f(\gamma(z)) - \frac{a_0^{(\gamma(x))}}{\abs{cr+dp}^k (\gamma(x) - \gamma(z))^k} \right) 
		(\gamma(x) - \gamma(z))^{k-2} \frac{dz}{(cz+d)^2}.
	\end{align}
	Since
	\[
	\gamma(x) - \gamma(z)
	=
	\frac{x-z}{(cx+d) (cz+d)},
	\]
	we have
	\begin{align}
		&\phant
		S_f (\gamma(x))
		\\
		&=
		\int_{z_0}^{\gamma^{-1} (\iu \infty)} \left( f(z) - \frac{a_0}{(cz+d)^k} \right) \left( \frac{x-z}{cx+d} \right)^{k-2} dz
		+ \int_{x}^{z_0} \left( f(z) - \frac{a_0^{(\gamma(x))} (cx+d)^k}{\abs{cr+dp}^k (x-z)^k} \right) 
		\left( \frac{x-z}{cx+d} \right)^{k-2} dz
		\\
		&\phant
		+ \frac{a_0}{k-1} \left( \frac{x-z_0}{(cx+d)(cz_0 + d)} \right)^{k-1} 
		+ \frac{a_0^{(\gamma(x))}}{\abs{cr+dp}^k} \frac{(cx+d) (cz_0 + d)}{x - z_0}.
	\end{align}
	
	Let $ \delta = \smat{r & r' \\ p & p'} \in \SL_2(\Z) $ and $ N_x \in \Z_{>0} $ be as defined before \zcref{thm:reciprocity_gen'd_Ded_sum}.
	Since
	\begin{align}
		\gamma \delta &= \pmat{ar+bp & ar'+bp' \\ cr+dp & cr'+dp'}, 
		\\
		(\eval{f}_k \gamma \delta) (\tau)
		&= (\eval{f}_k \delta) (\tau)
		= a_0^{(x)} + O \left( q^{1/N_x} \right),
	\end{align}
	we have $ a_0^{(\gamma(x))} = \sgn(cr+dp)^k a_0^{(x)} $.
	Thus, we have
	\begin{align}
		&\phant
		(\eval{S_f}_{2-k} \gamma) (x)
		\\
		&=
		\int_{z_0}^{\gamma^{-1} (\iu \infty)} \left( f(z) - \frac{a_0}{(cz+d)^k} \right) (x-z)^{k-2} dz
		+ \int_{x}^{z_0} \left( f(z) - \frac{a_0^{(x)}}{p^k (x-z)^k} \right) (x-z)^{k-2} dz
		\\
		&\phant
		+ \frac{a_0}{k-1} \frac{1}{cx+d} \left( \frac{x-z_0}{cz_0 + d} \right)^{k-1}  
		+ \frac{a_0^{(x)}}{p^k (cx+d)} \frac{cz_0 + d}{x - z_0}.
	\end{align}
	Thus, by \zcref{rem:Ded_sum_rep,lem:rep_period_poly}, we have
	\begin{align}
		&\phant
		(\eval{S_f}_{2-k} \gamma) (x)- S_f(x)
		\\
		&=
		\int_{z_0}^{\gamma^{-1} (\iu \infty)} \left( f(z) - \frac{a_0}{(cz+d)^k} \right) (x-z)^{k-2} dz
		- \int_{z_0}^{\iu \infty} \left( f(z) - a_0 \right) (x-z)^{k-2} dz
		\\
		&\phant
		+ \frac{a_0}{k-1} (x-z_0)^{k-1}\left( \frac{1}{(cx+d) (cz_0 + d)^{k-1}} - 1 \right)   
		+ \frac{a_0^{(x)}}{p^k} \frac{1}{x - z_0} \left( \frac{cz_0 + d}{cx+d} - 1 \right)
		\\
		&=
		R_{f, \gamma} (x)
		- \frac{a_0^{(x)}}{p^k} \frac{c}{cx+d}.
	\end{align}
\end{proof}


\section{Eichler Integrals and Generalized Dedekind Sums} \label{sec:Eichler_int}


In this section, we give another proof for reciprocity (\zcref{thm:reciprocity_gen'd_Ded_sum}) by considering asymptotics of the Eichler integral.

\begin{dfn}
	We define the \emph{Eichler integral of $ f $} as
	\[
	\calE_f (\tau) \coloneqq
	\int_{\tau}^{\iu \infty} (f(z) - a_0) (\tau - z)^{k-2} dz.
	\]
\end{dfn}

\begin{rem} \label{rem:Eichler_int_Fourier}
	We have
	\[
	\calE_f (\tau) =
	- \left( \frac{N}{2\pi\iu} \right)^{k-1} (k-2)!
	\sum_{n=1}^\infty \frac{a_n}{n^{k-1}} q^{n/N}.
	\]
\end{rem}

The Eichler integral satisfies a modular transformation formula of weight $ 2-k $ with an error term as the regularized period polynomial.

\begin{lem} \label{lem:Eichler_int_mod_trans}
	For any $ \gamma \in \Gamma $, we have
	\[
	(\eval{\calE_f}_{2-k} \gamma) (\tau) - \calE_f (\tau) 
	=
	R_{f, \gamma} (\tau).
	\]
\end{lem}

\begin{proof}
	Let $ \gamma = \smat{a & b \\ c & d} $.
	We have
	\begin{align}
		\calE_f (\gamma(\tau))
		&=
		\int_{\tau}^{\gamma^{-1} (\iu \infty)} (f(\gamma(z)) - a_0) (\gamma(\tau) - \gamma(z))^{k-2} d \gamma(z)
		\\
		&=
		(c\tau + d)^{2-k}
		\int_{\tau}^{\gamma^{-1} (\iu \infty)} \left( f(\gamma(z)) - \frac{a_0}{(cz+d)^k} \right) 
		(\tau - z)^{k-2} dz.
	\end{align}
	On the other hand, by \zcref{lem:rep_period_poly} with $ z_0 = \tau $, we have
	\begin{align}
		R_{f, \gamma} (\tau)
		&=
		\int_{\iu \infty}^{\tau} \left( f(z) - a_0 \right) (\tau - z)^{k-2} dz
		+ \int_{\tau}^{\gamma^{-1} (\iu \infty)} \left( f(z) - \frac{a_0}{(cz+d)^k} \right) (\tau - z)^{k-2} dz
		\\
		&=
		(\eval{\calE_f}_{2-k} \gamma) (\tau) - \calE_f (\tau).
	\end{align}
\end{proof}

Our generalized Dedekind sums appear in the asymptotics of the Eichler integral.

\begin{prop} \label{prop:Eichler_int_Ded_sum}
	We have
	\[
	\calE_f (\tau)
	=
	\frac{a_0^{(x)}}{(k-1) p^k} \frac{1}{\tau - x} + S_f(x) + O(\tau - x)
	\quad \text{ as } \tau \to x.
	\]
\end{prop}

\begin{proof}
	For arbitrary $ z_0 \in \bbH $, we have
	\begin{align}
		\calE_f (\tau + x)
		&=
		\int_{x}^{\iu \infty} (f(z + \tau) - a_0) (x - z)^{k-2} dz
		\\
		&=
		\int_{z_0}^{\iu \infty} (f(z + \tau) - a_0) (x - z)^{k-2} dz
		+ \int_{x}^{z_0} \left( f(z + \tau) - \frac{a_0^{(x)}}{p^k (x-z-\tau)^k} \right) (x - z)^{k-2} dz
		\\
		&\phant
		+ \int_{x}^{z_0} \left( \frac{a_0^{(x)}}{p^k (x-z-\tau)^k} - a_0 \right) (x - z)^{k-2} dz.
	\end{align}
	Since
	\[
	\frac{d}{dz} \left( \frac{x-z}{x-z-\tau} \right)^{k-1}
	=
	(k-1) \tau \frac{(x-z)^{k-2}}{(x-z-\tau)^k},
	\]
	the last integral can be calculated as
	\begin{align}
		&\phant
		\frac{a_0^{(x)}}{(k-1) p^k} \tau^{-1} \left( \frac{x-z_0}{x - z_0 - \tau} \right)^{k-1} + \frac{a_0}{k-1} (x - z_0)^{k-1}
		\\
		&=
		\frac{a_0^{(x)}}{(k-1) p^k} \tau^{-1} + \frac{a_0^{(x)}}{p^k} \frac{1}{x - z_0} + \frac{a_0}{k-1} (x - z_0)^{k-1}
		+ O(\tau) \quad \text{ as } \tau \to 0.
	\end{align}
	Since $ f'(z) = (2\pi\iu/N) \sum_{n=1}^\infty a_n n \bm{e} (nz/N) $ converges uniformly, for any $ y_0 > 0 $, there exists $ K>0 $ such that for all $ z \in \bbH $ with $ \Im(z) \ge y_0 $, it holds that $ \abs{f'(z)} \le K e^{-2\pi \Im(z)/N} $.
	Thus, if $ \Im(z) \ge y_0 $, then we have
	\[
	\abs{f(z + \tau) - f(z)}
	= \abs{\tau \int_0^1 f'(z + \tau u) du}
	\le K \abs{\tau}.
	\]
	Therefore, we obtain
	\begin{align}
		\calE_f (\tau + x)
		&=
		\int_{z_0}^{\iu \infty} (f(z) - a_0) (x - z)^{k-2} dz
		+ \int_{x}^{z_0} \left( f(z) - \frac{a_0^{(x)}}{p^k (x-z)^k} \right) (x - z)^{k-2} dz
		\\
		&\phant
		+ \frac{a_0^{(x)}}{(k-1) p^k} \tau^{-1} + \frac{a_0^{(x)}}{p^k} \frac{1}{x - z_0} + \frac{a_0}{k-1} (x - z_0)^{k-1}
		+ O(\tau) \quad \text{ as } \tau \to 0.
	\end{align}
	By \zcref{rem:Ded_sum_rep}, we obtain the claim.
\end{proof}

By combining \zcref{lem:Eichler_int_mod_trans,prop:Eichler_int_Ded_sum}, we obtain another proof of reciprocity (\zcref{thm:reciprocity_gen'd_Ded_sum}).

\begin{proof}[An alternative proof of \zcref{thm:reciprocity_gen'd_Ded_sum}]
	We take the limit as $ \tau \to x $ of
	\[
	(\eval{\calE_f}_{2-k} \gamma) (\tau) - \calE_f (\tau) 
	=
	R_{f, \gamma} (\tau).
	\]
	By \zcref{prop:Eichler_int_Ded_sum} and the fact $ a_0^{(\gamma(x))} = \sgn(cr+dp)^k a_0^{(x)} $ shown in the above proof of \zcref{thm:reciprocity_gen'd_Ded_sum}, we have
	\begin{align}
		\calE_f (\gamma(\tau))
		&=
		\frac{a_0^{(\gamma(x))}}{(k-1) \abs{cr+dp}^k} \frac{1}{\gamma(\tau) - \gamma(x)}
		+ S_f (\gamma(x)) + O(\gamma(\tau) - \gamma(x))
		\\
		&=
		\frac{a_0^{(x)}}{(k-1) (cr+dp)^k} \frac{(c\tau + d) (cx+d)}{\tau - x}
		+ S_f (\gamma(x)) + O(\tau - x)
		\\
		&=
		\frac{a_0^{(x)}}{(k-1) p^k (cx+d)^k} \left( \frac{(cx+d)^2}{\tau - x} + c(cx+d) \right)
		+ S_f (\gamma(x)) + O(\tau - x)
		\quad \text{ as } \tau \to x.
	\end{align}
	On the other hand, we have
	\[
	(c\tau + d)^{k-2}
	=
	(cx + d)^{k-2} + (k-2)c(cx+ d)^{k-3} (\tau - x)
	+ O((\tau - x)^2) \quad \text{ as } \tau \to x.
	\]
	Thus, we obtain
	\[
	(\eval{\calE_f}_{2-k} \gamma) (\tau)
	=
	\frac{a_0^{(x)}}{(k-1) p^k} \frac{1}{\tau - x}
	+ (cx + d)^{k-2} S_f (\gamma(x)) 
	+ \frac{a_0^{(x)}}{p^k} \frac{c}{cx+d}
	+ O(\tau - x) \quad \text{ as } \tau \to x.
	\]
	Therefore, we obtain 
	\begin{align}
		R_{f, \gamma} (\tau)
		&=
		(\eval{\calE_f}_{2-k} \gamma) (\tau) - \calE_f (\tau)
		\\
		&=
		(\eval{S_f}_{2-k} \gamma) (x) - S_f (x)
		+ \frac{a_0^{(x)}}{p^k} \frac{c}{cx+d}
		+ O(\tau - x) \quad \text{ as } \tau \to x,
	\end{align}
	which implies the claim.
\end{proof}


\section{Asymptotic formulas} \label{sec:asymp}


In this section, we give another proof of the asymptotic formula in \zcref{prop:Eichler_int_Ded_sum} by using a technique of Mellin transform.


\subsection{Asymptotic expansions and Mellin transform} \label{subsec:asymp_Mellin}


We present an asymptotic formula that follows from the Mellin transform.

\begin{dfn}[Asymptotic expansion] \label{dfn:asymptotic}
	For a function $ \varphi \colon (0, \infty) \to \bbC $,
	a monotonically increasing sequence $ (\nu_j)_{j=0}^\infty $ in $ \R_{>0} $,
	and a sequence $ (a_j)_{j=0}^\infty $ in $ \bbC $, we write 
	\[
	\varphi(t) \sim \sum_{j=0}^{\infty} a_j t^{-\nu_j} \text{ as } t \to +0
	\]
	to express that
	\[
	\varphi(t) = \sum_{j=0}^{j_1} a_j t^{-\nu_j} + O(t^{-\nu_{j+1}}) \text{ as } t \to +0
	\]
	holds for any integer $ j \ge 0 $.
	
	In this case, we call the above equation the \emph{asymptotic expansion of $ \varphi(t) $ as $ t \to +0 $}. 
\end{dfn}

\begin{dfn}[Mellin transform]
	For a continuous function $ \varphi \colon (0, \infty) \to \bbC $ which satisfies 
	\[
	\varphi (t) =
	\begin{cases}
		O(t^{-\alpha}) & \text{ as $ t \to +0 $}, \\
		O(t^{-\beta}) & \text{ as $ t \to +\infty $}
	\end{cases}
	\]
	for some $ \alpha < \beta $, we define its Mellin transform as
	\[
	\calM \varphi (s) = \calM [\varphi] (s)
	\coloneqq
	\int_{0}^{\infty} \varphi(t) t^{s-1} dt.
	\]
\end{dfn}

It is known that $ \calM \varphi (s) $ is a holomorphic function on $ (\alpha, \beta) + \iu \R $.

It is well known that the asymptotic expansion of $ \varphi(t) $ is governed by the poles of its Mellin transform $ \calM \varphi (s) $ (\cite[Theorem 3 and 4]{Flajolet-Gourdon-Dumas}). 
From this relation, the following asymptotic formula follows.

\begin{thm}[Mellin summation formula, {\cite[Theorem 5]{Flajolet-Gourdon-Dumas}}] \label{thm:Mellin_sum_formula}
	Let $ \varphi \colon (0, \infty) \to \bbC $ be a continuous function and
	\[
	\Lambda (s) \coloneqq
	\sum_{n=1}^\infty \frac{\lambda_n}{\mu_n^s}
	\]
	be a Dirichlet series defined by sequences $ (\lambda_n)_{n=1}^\infty $ in $ \bbC $ and $ (\mu_n)_{n=1}^\infty $ in $ \R_{>0} $.
	We assume the following conditions:
	\begin{itemize}
		\item There exist $ \alpha < \beta $ such that 
		\[
		\varphi(t) =
		\begin{cases}
			O(t^{-\alpha}) & \text{ as $ t \to +0 $}, \\
			O(t^{-\beta}) & \text{ as $ t \to +\infty $}.
		\end{cases}
		\]
		\item The series $ \Lambda(s) $ converges absolutely in the region $ \Re(s) > \sigma_0 $ for some $ \sigma_0 < \beta $.
		\item Functions $ \calM \varphi(s) $ and $ \Lambda(s) $ extend on $ \bbC $ meromorphically.
		\item In the region $ \Re(s) < \beta $, the following hold:
		\begin{itemize}
			\item We have $ \calM \varphi(s) = O(\abs{s}^{-R}) $ as $ \Im(s) \to \pm \infty $ for any $ R > 0 $.
			\item The series $ \Lambda(s) $ has at most polynomial growth as $ \Im(s) \to \pm \infty $.
		\end{itemize}
	\end{itemize}
	We denote by $ \calP $ the set of poles of $ \Lambda(s) \calM \varphi(s) $ in the region $ \Re(s) < \beta $.
	For a pole $ \xi \in \calP $, we denote the principal part of $ \Lambda(s) \calM \varphi(s) $ at $ s = \xi $ as%
	\footnote{
		Here, the notation $l \ll \infty$ means that there exists a constant $C$ such that $l < C$.
	}
	\[
	\Lambda(s) \calM \varphi(s)
	=
	\sum_{0 \le l \ll \infty} \frac{c_{\xi, l}}{(s - \xi)^{l+1}} + O(1)
	\quad \text{ as $ s \to \xi $}.
	\]
	Under the above assumptions, we have the following asymptotic formula:
	\[
	\sum_{n=1} \lambda_n \varphi(\mu_n t)
	\sim
	\sum_{\xi \in \calP} \sum_{0 \le l \ll \infty} \frac{(-1)^l}{l!} c_{\xi, l} t^{-\xi} (\log t)^l
	\quad
	\text{ as $ t \to +0 $}.
	\]
\end{thm}

We modify this formula in a form suitable for our purposes.

\begin{dfn}
	Let $ a \in \R $.
	A $ C^\infty $ function $ \varphi \colon (a, \infty) \to \bbC $ is called \emph{of rapid decay} as $ t \to +\infty $ 
	if $ t^m \varphi^{(n)} (x) $ is bounded as $ t \to +\infty $ for any $ m, n \in \Z_{\ge 0} $.
\end{dfn}

\begin{cor} \label{cor:asymp_Mellin}
	Let $ \varepsilon > 0 $, $ \varphi \colon (-\varepsilon, \infty) \to \bbC $ be a $ C^\infty $ function of rapid decay as $ t \to + \infty $, and
	\[
	\Lambda (s) \coloneqq
	\sum_{n=1}^\infty \frac{\lambda_n}{\mu_n^s}
	\]
	be a Dirichlet series defined by sequences $ (\lambda_n)_{n=1}^\infty $ in $ \bbC $ and $ (\mu_n)_{n=1}^\infty $ in $ \R_{>0} $.
	We assume that the series $ \Lambda(s) $ satisfies the following properties:
	\begin{itemize}
		\item it converges absolutely in the region $ \Re(s) > \sigma_0 $ for some $ \sigma_0 > 0 $\textup{;}
		\item it admits a meromorphic continuation on $ \bbC $\textup{;}
		\item it has at most polynomial growth as $ \Im(s) \to \pm \infty $\textup{;}
		\item all its poles are simple.
	\end{itemize}
	We denote by $ \calP $ the set of poles of $ \Lambda(s) $ in $ \bbC $.
	Then, we have an asymptotic formula
	\begin{align}
		\sum_{n=1}^\infty \lambda_n \varphi(\mu_n t)
		&\sim
		\sum_{\xi \in \calP \smallsetminus \Z_{\le 0}} 
		\left( \Res_{s = \xi} \Lambda(s) \right) \calM \varphi (\xi) t^{-\xi}
		\\
		&\phantom{{}\sim{}}
		+ \sum_{j \in \Z_{\ge 0} \cap (-\calP)} 
		\left(
		\frac{\varphi^{(j)} (0)}{j!} \left( \CT_{s = -j} \Lambda(s) - \log (t) \Res_{s = -j} \Lambda(s) \right)
		+ C_{\varphi, j} \Res_{s = -j} \Lambda(s)
		\right) t^j
		\\
		&\phantom{{}\sim{}}
		+ \sum_{j \in \Z_{\ge 0} \smallsetminus (-\calP)} \Lambda(-j) \frac{\varphi^{(j)} (0)}{j!} t^j
		\quad \text{ as } t \to +0,
	\end{align}
	where
	\begin{align}
		\Lambda(s) &= 
		\frac{1}{s - \xi} \Res_{s = \xi} \Lambda(s) + \CT_{s = \xi} \Lambda(s)
		+ O(s-\xi)
		\quad \text{ as $ s \to \xi $},
		\\
		C_{\varphi, j}
		&\coloneqq
		\int_{1}^{\infty} \varphi(t) t^{j-1} dt
		+ \int_{0}^{1}
		\left( \varphi(t) - \sum_{0 \le l \le j} \frac{\varphi^{(l) (0)} (0)}{l!} t^l \right) t^{j-1} dt
		+ \sum_{l=0}^{j-1} \frac{\varphi^{(l)}}{l!} \frac{1}{l-j}.
	\end{align}
\end{cor}

The proof relies on the following lemma.

\begin{lem} \label{lem:Mellin_smooth}
	For $ \varepsilon > 0 $ and a $ C^\infty $ function $ \varphi \colon (-\varepsilon, \infty) \to \bbC $ of rapid decay as $ t \to + \infty $, the following hold.
	\begin{enumerate}
		\item \label{item:lem:Mellin_smooth:conv}
		The Mellin transform $ \calM \varphi (s) $ converges absolutely on $ \Re(s) > 0 $.
		
		\item \label{item:lem:Mellin_smooth:eval}
		For any $ \sigma \in \R_{>0} $ and $ N \in \Z_{>0} $, it holds
		$ \calM \varphi (\sigma + \iu y) = O(\abs{y}^{-N}) $ as $ y \to \pm \infty $.
		
		\item \label{item:lem:Mellin_smooth:residue}
		The function $ \calM \varphi (s) $ extends meromorphically on $ \bbC $ and its possible poles are in $ \Z_{\le 0} $.
		For any $ j \in \Z_{\ge 0} $, we have
		\[
		\calM \varphi (s) =
		\frac{\varphi^{(j)} (0)}{j!} \frac{1}{s+j} + C_{\varphi, j} + O(s+j)
		\quad \text{ as } s \to -j.
		\]
	\end{enumerate}
\end{lem}

\begin{proof}
	\zcref{item:lem:Mellin_smooth:conv} follows from 
	\[
	\varphi(t) =
	\begin{cases}
		O(1) & \text{ as $ t \to +0 $}, \\
		O(t^{-N}) & \text{ as $ t \to +\infty $ for any $ N \in \Z_{\ge 0} $}.
	\end{cases}
	\]
	
	We prove \zcref{item:lem:Mellin_smooth:eval}.
	Fix $ \sigma \in \R_{>0} $ and $ N \in \Z_{>0} $ arbitrarily.
	We have
	\[
	\calM \varphi (\sigma + \iu y)
	=
	\int_{-\infty}^{\infty} f(e^u) e^{\sigma u} e^{\iu uy} du.
	\]
	This can be viewed as the Fourier transform.
	Since $ f(e^u) e^{\sigma u} $ is an element of the Sobolev space
	\[
	W^{N, 1} (\R)
	\coloneqq
	\left\{
		\psi \in L^1 (\R)
		\relmiddle|
		\sum_{j=0}^N \int_{\R} \abs{\psi^{(j)} (u)} du
		< \infty
	\right\},
	\]
	its Fourier transform $ \calM \varphi (\sigma + \iu y) $ is $ O(\abs{y}^N) $ as $ y \to \pm \infty $.
	
	Finally, we prove \zcref{item:lem:Mellin_smooth:residue}.
	For any $ j \in \Z_{\ge 0} $, we have
	\begin{align}
		\calM \varphi(s) 
		&=
		\int_{1}^{\infty} \varphi(t) t^{s-1} dt
		+ \int_{0}^{1} \left( \varphi(t) - \sum_{l=0}^j \frac{\varphi^{(l)} (0)}{l!} t^{l} \right)
		+ \sum_{l=0}^j \frac{\varphi^{(l)} (0)}{l!} \frac{1}{s+l}
		\\
		&=
		\frac{\varphi^{(j)} (0)}{j!} \frac{1}{s+j}
		+ C_{\varphi, j} + O(s+j)
		\quad \text{ as } s \to -j.
	\end{align}
\end{proof}

\begin{proof}[Proof of \zcref{cor:asymp_Mellin}]
	By \zcref{lem:Mellin_smooth}, all possible poles of $ \Lambda(s) \calM \varphi (s) $ lie on $ \calP \cup \Z_{\le 0} $ and 
	for each $ \xi \in \calP \smallsetminus \Z_{\le 0} $ and $ j \in \Z_{\ge 0} $, we have
	\begin{align}
		\Lambda(s) \calM \varphi (s)
		&=
		\frac{1}{s-\xi} \calM \varphi(\xi) \Res_{s = \xi} \Lambda(s)
		+ O(1) \quad \text{ as } s \to \xi,
		\\
		\Lambda(s) \calM \varphi (s)
		&=
		\frac{1}{(s+j)^2} \frac{\varphi^{(j)} (0)}{j!} \Res_{s = \xi} \Lambda(s)
		+ \frac{1}{s+j} \left( \frac{\varphi^{(j)} (0)}{j!} \CT_{s = \xi} \Lambda(s) + C_{\varphi, j} \Res_{s = \xi} \Lambda(s) \right)
		+ O(1) \quad \text{ as } s \to -j.
	\end{align}
	Thus, we obtain the claim by \zcref{thm:Mellin_sum_formula}.
\end{proof}


\subsection{Asymptotic expansions for sums in Fourier coefficients of modular forms} \label{subsec:asymp_mod_form}


\zcref{cor:asymp_Mellin} implies the following asymptotic formula.

\begin{cor} \label{cor:asymp_mod_form}
	For any $ \varepsilon > 0 $ and $ \varphi \colon \{ \tau \in \bbC \mid \Im(\tau) > -\varepsilon \} \to \bbC $ be a holomorphic function of rapid decay as $ \Im(\tau) \to + \infty $, we have
	\begin{align}
		\sum_{n=1}^\infty \frac{a_n}{n^{k-1}} \bm{e} \left( \frac{nx}{N} \right) \varphi(n\tau)
		&\sim
		-\left( \frac{2\pi\iu}{N} \right)^{k-1} \frac{1}{(k-2)!}
		\left(
			\frac{a_0^{(x)}}{(k-1) p^k} \frac{2\pi}{N} \calM [\varphi(\iu t)] (1) \tau^{-1}
			+ S_f (x) \varphi(0)
		\right)
		\\
		&\phant
		+ \sum_{j=1}^\infty L_f (k-j-1; x) \frac{\varphi^{(j)} (0)}{j!} \tau^j
		\quad \text{ as } \tau \to 0.
	\end{align}
\end{cor}

\begin{proof}
	Let $ 0 < \theta < \pi $ be arbitrary and put $ \tau = t e^{\iu \theta} $.
	We apply \zcref{cor:asymp_Mellin} for $ \varphi(t) $ as $ \varphi_\theta^{} (t) \coloneqq \varphi(t e^{\iu \theta})$ and $ \Lambda(s) $ as
	\begin{align}
		L_f (s+k-1; x)
		=
		-\left( \frac{2\pi}{N} \right)^{s+k-1} 
		\bm{e} \left( \frac{s+k-1}{4} \right)
		\frac{1}{\Gamma(s+k-1)}
		\widehat{L}_f (s+k-1; x).
	\end{align}
	By \zcref{lem:twisted_L-func}, $ \widehat{L}_f (s+k-1; x) $ has possible poles at $ s = 1, 1-k $ which are single.
	We also have $ \Gamma(s+k-1) $ has no zeros and single poles at $ s \in \Z_{\le 1-k} $.
	Thus, $ L_f (s+k-1; x) $ has possible poles at $ s = 1 $ which are single with residue
	\[
	\frac{a_0^{(x)}}{(k-1)! p^k} \left( \frac{2\pi\iu}{N} \right)^k
	\]
	by \zcref{lem:twisted_L-func}.
	We have
	\[
	\calM \varphi_\theta^{} (1) = \iu e^{-\iu \theta} \calM [\varphi(\iu t)] (1), \quad
	\quad
	\varphi_\theta^{(j)} (0) = e^{j \iu \theta} \varphi^{(j)} (0) 
	\]
	Thus, we obtain an asymptotic expansion
	\begin{align}
		\sum_{n=1}^\infty \frac{a_n}{n^{k-1}} \bm{e} \left( \frac{nx}{N} \right) \varphi(n\tau)
		&\sim
		\left( \frac{2\pi\iu}{N} \right)^{k-1} \frac{1}{(k-2)!}
		\left(
		\frac{a_0^{(x)}}{(k-1) p^k} \frac{2\pi\iu}{N} \calM [\varphi(\iu t)] (1) \frac{\iu}{t e^{\iu \theta}}
		- S_f (x) \varphi(0)
		\right)
		\\
		&\phant
		+ \sum_{j=1}^\infty L_f (k-j-1; x) \frac{\varphi^{(j)} (0)}{j!} \left( t e^{\iu \theta} \right)^j
		\quad \text{ as } t \to +0.
	\end{align}
	Since each asymptotic coefficient is independent of $ \theta $, we obtain the claim.
\end{proof}

Finally, we give another proof of the asymptotic formula in \zcref{prop:Eichler_int_Ded_sum}.

\begin{proof}[An alternative proof of \zcref{prop:Eichler_int_Ded_sum}]
	Let $ \varphi(\tau) \coloneqq \bm{e} (\tau/N) $.
	Then, we have
	\[
	\calE_f (x + \tau) =
	- \left( \frac{N}{2\pi\iu} \right)^{k-1} (k-2)!
	\sum_{n=1}^\infty \frac{a_n}{n^{k-1}} \bm{e} \left( \frac{nx}{N} \right) \varphi(n\tau),
	\]
	\[
	\calM [\varphi(\iu t)] (1) =
	\frac{N}{2\pi},
	\quad
	\varphi^{(j)} (0) =
	\left( \frac{2\pi\iu}{N} \right)^j.
	\]
	Thus, by \zcref{cor:asymp_mod_form}, we have
	\begin{align}
		\calE_f (x + \tau) 
		&=
		\frac{a_0^{(x)}}{(k-1) p^k} \tau^{-1}
		+ S_f (x)
		- (k-2)! \sum_{j=1}^\infty \frac{L_f (k-j-1; x)}{j!} \left( \frac{2\pi\iu}{N} \right)^{j + k - 1} \tau^j
		\quad \text{ as } \tau \to 0.
	\end{align}
\end{proof}


\section{Generalized Dedekind sums for Eisenstein series} \label{sec:Eisenstein}


We now specialize our generalized Dedekind sums to the case of Eisenstein series of $ \Gamma(N) $.


\subsection{Preliminaries} \label{subsec:prelim}



\subsubsection{Discrete Fourier transform} 


\begin{dfn} \label{dfn:discrete_Fourier_trans}
	Let $ N $ be a positive integer.
	For a periodic map $ \chi \colon \Z/N\Z \to \bbC $,
	we define its \emph{discrete Fourier transform} $ \widehat{\chi} \colon \Z/N\Z \to \bbC $ as
	\[
	\widehat{\chi} (n) \coloneqq
	\frac{1}{N} \sum_{m \in \Z/N\Z} \chi(m) \bm{e} \left( -\frac{mn}{N} \right).
	\]
\end{dfn}

\begin{lem}[Discrete Fourier expansion] \label{lem:discrete_Fourier_expansion}
	Let $ N $ be a positive integer and $ \chi \colon \Z/N\Z \to \bbC $ be a periodic map.
	Then, for any $ m \in \Z/N\Z $, we have
	\[
	\chi(m) = \sum_{n \in \Z/N\Z} \widehat{\chi} (n) \bm{e} \left( \frac{mn}{N} \right).
	\]
\end{lem}


\subsubsection{Periodic Bernoulli polynomials} 


We recall some basic facts about periodic Bernoulli polynomials.

\begin{dfn} \label{dfn:Bernoulli_poly}
	\begin{enumerate}
		\item We define the \emph{Bernoulli numbers} $ (B_j)_{j=0}^\infty $ as
		\[
		\sum_{j=0}^\infty \frac{B_j}{j!} t^j \coloneqq \frac{t }{e^t - 1}.
		\]
		
		\item We define the \emph{Bernoulli polynomials} $ (B_j(\alpha))_{j=0}^\infty $ as
		\[
		\sum_{j=0}^\infty \frac{B_j (\alpha)}{j!} t^j \coloneqq \frac{t e^{\alpha t}}{e^t - 1}.
		\]
		
		\item For an integer $ j \ge 0 $, we define the \emph{$ j $-th periodic Bernoulli polynomial} as
		\[
		\widetilde{B}_j (\alpha) \coloneqq B_j (\alpha - \floor{\alpha}).
		\]
	\end{enumerate}
\end{dfn}

\begin{lem} \label{lem:Bernoullli_poly_sym}
	For any $ j \ge 0 $, it holds 
	\[
	B_j (1-\alpha) = (-1)^j B_j(\alpha).
	\]
\end{lem}

The following Fourier expansions are known.

\begin{lem} \label{lem:periodic_Bernoullli_poly_Fourier_exp}
	For any $ j \ge 1 $ and $ \alpha \in \R $, it holds 
	\[
	\widetilde{B}_j (\alpha)
	=
	-\frac{j!}{(2\pi\iu)^j} \sideset{}{^*}\sum_{n \in \Z \smallsetminus \{ 0 \}} \frac{\bm{e} (n\alpha)}{n^j}
	- \frac{\delta_{j, 1}}{2} \bm{1}_{\Z} (\alpha), 
	\]
	where $ \bm{1}_{\Z} $ is the characteristic function of $ \Z $ and
	\[
	\sideset{}{^*}\sum_{n \in \Z \smallsetminus \{ 0 \}}
	\coloneqq
	\begin{dcases}
		\sum_{n \in \Z \smallsetminus \{ 0 \}} & \text{ if } j \ge 2, \\
		\lim_{R \to \infty} \sum_{n \in \Z \smallsetminus \{ 0 \}, \, \abs{n} \le R} & \text{ if } j = 1.
	\end{dcases}
	\]
\end{lem}

Periodic Bernoulli polynomials have the following discrete Fourier expansions.

\begin{lem} \label{lem:periodic_Bernoullli_poly_disc_Fourier_exp}
	Let $ j $ be a positive integer.
	For any positive integer $ N $ and $ m, n \in \Z/N\Z $, we have
	\begin{align}
		\sum_{m \in \Z/N\Z} \widetilde{B}_j \left( \frac{m}{N} \right) \bm{e} \left( \frac{mn}{N} \right)
		&=
		\begin{dcases}
			\frac{B_j}{N^{j-1}}  & \text{ if } N \mid n, \\
			\frac{j}{(2\iu)^j N^{j-1}} \cot^{(j-1)} \left( \frac{\pi n}{N} \right) - \frac{\delta_{j, 1}}{2}
			 & \text{ if } N \nmid n.
		\end{dcases}
		\\
		\widetilde{B}_j \left( \frac{m}{N} \right)
		&=
		\frac{B_j}{N^j} 
		+ \frac{\delta_{j, 1}}{2N}
		+ \frac{j}{(-2N \iu)^j}
		\sum_{1 \le n < N} \cot^{(j-1)} \left( \frac{\pi n}{N} \right) \bm{e} \left( \frac{mn}{N} \right).
	\end{align}
\end{lem}

\begin{proof}
	By \zcref{lem:discrete_Fourier_expansion}, it suffices to show only the first equality.
	The case $ N \mid n $ follows from the multiplicative formula for the Bernoulli polynomials.
	We assume $ N \nmid n $.
	By \zcref{lem:periodic_Bernoullli_poly_Fourier_exp}, we have
	\begin{align}
		\sum_{m \in \Z/N\Z} \widetilde{B}_j \left( \frac{m}{N} \right) \bm{e} \left( \frac{mn}{N} \right)
		&=
		\sum_{m \in \Z/N\Z} 
		\frac{-j!}{(2\pi\iu)^j} \sideset{}{^*}\sum_{l \in \Z \smallsetminus \{ 0 \}} \frac{1}{l^j}
		\bm{e} \left( \frac{(l+n)m}{N} \right)
		- \frac{\delta_{j, 1}}{2}
		\\
		&=
		-\frac{j! N}{(2\pi\iu)^j} 
		\sideset{}{^*}\sum_{\substack{
				l \in \Z \smallsetminus \{ 0 \}, \\
				l \equiv -n \bmod N
		}} 
		\frac{1}{l^j}
		- \frac{\delta_{j, 1}}{2}
		\\
		&=
		\frac{(-1)^{j-1} j!}{(2\pi\iu)^j N^{j-1}} 
		\sideset{}{^*}\sum_{l \in \Z} \frac{1}{(l + n/N)^j}
		- \frac{\delta_{j, 1}}{2}
		\\
		&=
		\frac{j}{(2\iu)^j N^{j-1}} \cot^{(j-1)} \left( \frac{\pi n}{N} \right)
		- \frac{\delta_{j, 1}}{2}.
	\end{align}
\end{proof}

%
%


\subsubsection{Hurwitz zeta function} 


\begin{dfn}
	For complex numbers $ s $ and $ \alpha $ such that $ \Re(s) > 1 $ and $ \Re(\alpha) > 0 $, we define the \emph{Hurwitz zeta function} as
	\begin{equation} \label{eq:Hurwitz_zeta}
		\zeta(s; \alpha) \coloneqq
		\sum_{n=0}^\infty \frac{1}{(n + \alpha)^s}.
	\end{equation}
\end{dfn}

It is known that for fixed $ \alpha $, the Hurwitz zeta function $ \zeta(s; \alpha) $ extends holomorphically to $ \bbC \smallsetminus \{ 1 \} $ and has a single pole at $ s=1 $,

\begin{lem} \label{lem:Hurwitz_zeta_special_value}
	For fixed $ \alpha $ and any $ j \in \Z_{\ge 0} $, we have
	\[
	\zeta(-j; \alpha) = -\frac{B_{j+1} (\alpha)}{j+1}.
	\]
\end{lem}


\subsection{Basic facts for Eisenstein series of $ \Gamma(N) $} \label{subsec:Eisenstein:basic}


In this subsection, we briefly recall the definition and basic properties of Eisenstein series of $ \Gamma(N) $.
As before, we fix integers $ k \ge 2 $ and $ N \ge 1 $.

%

\begin{dfn}
	\label{dfn:Eisenstein_vector}
	For $ v = (c_v, d_v) \in \Z^2 / N \Z^2 $, we define
	\[
	E_k^v (\tau)
	\coloneqq
	(k-1)! \left( -\frac{N}{2\pi\iu} \right)^k 
	\sideset{}{^*}\sum_{(c, d) \in ((c_v, d_v) + N\Z^2) \smallsetminus \{ (0, 0) \}}
	\frac{1}{(c \tau + d)^k},
	\]
	where
	\[
	\sideset{}{^*}\sum_{(c, d) \in ((c_v, d_v) + N\Z^2) \smallsetminus \{ (0, 0) \}}
	\coloneqq
	\begin{dcases}
		\sum_{(c, d) \in ((c_v, d_v) + N\Z^2) \smallsetminus \{ (0, 0) \}} & \text{ if } k \ge 3, \\
		\sum_{c \in c_v + N\Z} \sum_{\substack{d \in d_v + N\Z, \\ (c, d) \neq (0, 0)}} & \text{ if } k = 2.
	\end{dcases}
	\]
\end{dfn}

\begin{prop} \label{prop:Eisenstein_vector}
	Let $ v = (c_v, d_v) \in \Z^2 / N \Z^2 $.
	\begin{enumerate}
		\item \label{item:prop:Eisenstein_vector:mod_form}
		\textup{(}\cite[Corollary 4.2.2 and Equation 4.5]{Diamond-Shurman}\textup{)}
		The function $ E_k^v (\tau) $ is a modular form for $ \Gamma(N) $ of weight $ k $.
		
		\item \label{item:prop:Eisenstein_vector:Fourier_exp}
		We have
		\[
		E_k^v (\tau)
		=
		\frac{(-N)^{k-1}}{k} \bm{1}_{N\Z}^{} (c_v) 
		\sum_{m \in \Z/N\Z} \bm{e} \left( -\frac{d_v m}{N} \right) \widetilde{B}_k \left( \frac{m}{N} \right)
		+ \sum_{n=1}^\infty \left( \sigma_{k-1}^v (n) + (-1)^k \sigma_{k-1}^{-v} (n) \right) q^{n/N},
		\]
		where $ \bm{1}_{N\Z}^{} (m) $ is the characteristic function of $ N\Z $ and
		\[
		\sigma_{k-1}^v (n)
		\coloneqq
		\sum_{\substack{
			0 < d \mid n, \\
			n/d \equiv c_v \bmod N
		}}
		\bm{e} \left( \frac{d_v d}{N} \right) d^{k-1}.
		\]
		
		\item \label{item:prop:Eisenstein_vector:cusp} \textup{(}\cite[Proposition 4.2.1 and Equation (4.6)]{Diamond-Shurman}\textup{)}
		For any $ \gamma = \smat{a & b \\ c & d} \in \SL_2 (\Z) $, we have
		\[
		\eval{E_k^v}_k \gamma = E_k^{v \gamma}.
		\]
		In particular, for each cusp $ a/c $, we have
		\[
		a_{E_k^v, 0}^{(a/c)}
		=
		\frac{(-N)^{k-1}}{k} \bm{1}_{N\Z}^{} (a c_v + c d_v) 
		\sum_{m \in \Z/N\Z} \bm{e} \left( -\frac{(b c_v + d d_v) m}{N} \right) \widetilde{B}_k \left( \frac{m}{N} \right),
		\]
		where $ a_{E_k^v, 0}^{(a/c)} $ is defined before \zcref{thm:reciprocity_gen'd_Ded_sum}.
		
		
	\end{enumerate}
\end{prop}

\begin{proof}[Proof of \zcref{item:prop:Eisenstein_vector:Fourier_exp}]
	By \cite[Theorem 4.2.3]{Diamond-Shurman}, we have
	\[
	E_k^v (\tau)
	=
	(k-1)! \left( -\frac{N}{2\pi\iu} \right)^k
	\bm{1}_{N\Z}^{} (c_v) \sum_{d \in (d_v + N\Z) \smallsetminus \{ 0 \}} \frac{1}{d^k}
	+ \sum_{n=1}^\infty \left( \sigma_{k-1}^v (n) + (-1)^k \sigma_{k-1}^{-v} (n) \right) q^{n/N}.
	\]
	We have
	\[
	\sum_{d \in (d_v + N\Z) \smallsetminus \{ 0 \}} \frac{1}{d^k}
	=
	\frac{1}{N} \sum_{d \in \Z \smallsetminus \{ 0 \}} \frac{1}{d^k}
	\sum_{m \in \Z/N\Z} \bm{e} \left( \frac{(d-d_v)m}{N} \right).
	\]
	By \zcref{lem:periodic_Bernoullli_poly_Fourier_exp}, this is equal to
	\[
	-\frac{(2\pi\iu)^k}{k! N} 
	\sum_{m \in \Z/N\Z} \bm{e} \left( -\frac{d_v m}{N} \right) \widetilde{B}_k \left( \frac{m}{N} \right).
	\]
\end{proof}


\subsection{Generalized Dedekind sums for Eisenstein series of $ \Gamma(N) $ with periodic maps} \label{subsec:Eisenstein:char_&_Ded_sum}


Based on the above preparations, we define Eisenstein series with periodic maps and state the reciprocity formula for the associated generalized Dedekind sums.

\begin{dfn} \label{dfn:Eisenstein_period_map}
	For periodic maps $ \chi, \psi \colon \Z/N\Z \to \bbC $ (not necessarily Dirichlet characters), we define
	\[
	E_k^{\chi, \psi} (\tau)
	\coloneqq
	\sum_{v = (c_v, d_v) \in \Z^2 / N\Z^2} \chi(c_v) \widehat{\psi}(d_v) E_k^v (\tau).
	\]
\end{dfn}

\begin{prop} \label{prop:Eisenstein_period_map}
	Let $ \chi, \psi \colon \Z/N\Z \to \bbC $ be periodic maps.
	\begin{enumerate}
		\item \label{item:prop:Eisenstein_period_map:Fourier_exp}
		We have
		\[
		E_k^{\chi, \psi} (\tau)
		=
		\frac{(-N)^{k-1}}{k} \chi(0) \sum_{m \in \Z/N\Z} \psi(m) \widetilde{B}_k \left( \frac{m}{N} \right)
		+ \sum_{n=1}^\infty \left( \sigma_{k-1}^{\chi, \psi} (n) + (-1)^k \sigma_{k-1}^{\chi^-, \psi^-} (n) \right) q^{n/N},
		\]
		where $ \chi^- (m) \coloneqq \chi(-m) $ and
		\[
		\sigma_{k-1}^{\chi, \psi} (n)
		\coloneqq
		\sum_{0 < d \mid n} \chi \left( \frac{n}{d} \right) \psi(d) d^{k-1}.
		\]
		
		\item \label{item:prop:Eisenstein_period_map:cusp} 
		For each cusp $ a/c $ with coprime integers $ a \in \Z $ and $ c \in \Z_{>0} $, we have
		\[
		a_{E_k^{\chi, \psi}, 0}^{(a/c)}
		=
		\frac{(-N)^{k-1}}{k}
		\sum_{l, m \in \Z/N\Z} \chi(-cl) \widehat{\psi} (al)
		\bm{e} \left( -\frac{lm}{N} \right) \widetilde{B}_k \left( \frac{m}{N} \right).
		\]
	\end{enumerate}
\end{prop}

\begin{proof}
	\zcref{item:prop:Eisenstein_period_map:Fourier_exp} follows from \zcref{prop:Eisenstein_vector} \zcref{item:prop:Eisenstein_vector:Fourier_exp,lem:discrete_Fourier_expansion}.
	
	\zcref{item:prop:Eisenstein_period_map:cusp} follows from \zcref{prop:Eisenstein_vector} \zcref{item:prop:Eisenstein_period_map:cusp} and the fact that for any $ \gamma = \smat{a & b \\ c & d} \in \SL_2 (\Z) $, 
	\[
	\begin{array}{ccc}
		\{ v = (c_v, d_v) \in \Z^2 / N\Z^2 \mid a c_v + c d_v \equiv 0 \bmod N \} & \longrightarrow & \Z / N\Z \\
		v = (c_v, d_v) & \longmapsto & b c_v + d d_v \\
		(-cl, al) & \text{\reflectbox{$ \longmapsto $}} & l
	\end{array}
	\]
	is bijective.
\end{proof}

\begin{rem}
	For a congruence subgroup $ \Gamma \subset \SL_2(\Z) $, we define the \emph{weight $ k $ Eisenstein space of $ \Gamma $} as
	\[
	\calE_k(\Gamma)
	\coloneqq
	\left\{
	f \in M_k(\Gamma)
	\relmiddle|
	\sprod{f, g} = 0 \text{ for any } g \in S_k(\Gamma)
	\right\},
	\]
	where we denote by $ M_k(\Gamma) $ and $ S_k(\Gamma) $ the spaces of weight $ k $ modular forms and cusp forms of $ \Gamma $ respectively, let $ \tau = u + \iu y $ and we define the Petersson inner product as
	\[
	\sprod{f,g}
	\coloneqq
	\frac{1}{\vol(\Gamma(N) \backslash \bbH)}
	\int_{\Gamma(N) \backslash \bbH} f(\tau) \overline{g(\tau)} y^k \frac{du dy}{y^2}.
	\]
	Then, by \cite[Theorem 4.2.3, 4.5.2, and 4.6.2 and Section 5.11]{Diamond-Shurman}, we have
	\begin{align}
		\calE_k(\Gamma(N)) 
		&=
		\sprod{ E_k^v (\tau) \relmiddle| v \in \Z^2/N\Z^2 }
		=
		\sprod{ E_k^{\chi, \psi} (\tau) \relmiddle| \chi, \psi \colon \Z/N\Z \to \bbC },
		\\
		\calE_k(\Gamma_1(N)) 
		&=
		\sprod{ E_k^{\chi, \psi} (\tau) \relmiddle| \chi, \psi \colon \Z/N\Z \to \bbC \text{ are Dirichlet characters}},
		\\
		\calE_k(\Gamma_0(N)) 
		&=
		\sprod{ E_k^{\chi, \psi} (\tau) \relmiddle| \chi, \psi \colon \Z/N\Z \to \bbC \text{ are Dirichlet characters such that } 
			\chi \psi = \bm{1}_{(\Z/N\Z)^\times} },
	\end{align}
	where $ \bm{1}_{(\Z/N\Z)^\times}^{} (m) $ is the characteristic function of $ (\Z/N\Z)^\times $.
	
	We also remark that $ \calE_k(\Gamma(N)) \neq 0 $ if and only if either $ k $ is even or $ N \ge 3 $ by \cite[Equation (4.3)]{Diamond-Shurman}.
\end{rem}

 As before, we fix a rational number $ x = r/p \in \Q $ with coprime integers $ r \in \Z $ and $ p \in \Z_{>0} $.

\begin{prop} \label{prop:Eisenstein_period_map_L-func}
	Let $ \chi, \psi \colon \Z/N\Z \to \bbC $ be periodic maps.
	\begin{enumerate}
		\item \label{item:prop:Eisenstein_period_map_L-func:rep}
		We have
		\begin{align}
			&\phant
			L_{E_k^{\chi, \psi}}(s; x)
			\\
			&=
			N^{-s + k -1} p^{-2s + k -1} \sum_{1 \le m, n \le Np} 
			\left( \chi(m) \psi(n) + (-1)^k \chi(-m) \psi(-n) \right)
			\bm{e} \left( \frac{mnx}{N} \right) 
			\zeta \left( s-k+1; \frac{m}{Np} \right) \zeta \left( s; \frac{n}{Np} \right).
		\end{align}
		
		\item \label{item:prop:Eisenstein_period_map_L-func:special_value}
		For any $ 1 \le j \le k-1 $, we have
		\[
		\widehat{L}_{E_k^{\chi, \psi}}(j; x)
		=
		\widehat{L}_k^{\chi, \psi} (j; x) + \widehat{L}_k^{\chi, \psi, *} (j; x)
		\]
		where
		\begin{align}
			\widehat{L}_k^{\chi, \psi} (j; x)
			&\coloneqq
			-\frac{\iu^j N^{k-j-1}}{2^j (k-j)} p^{k-2j-1}
			\sum_{1 \le m, n \le Np-1} 
			\chi(m) \psi(n) 
			\bm{e} \left( \frac{mnx}{N} \right)
			B_{k-j} \left( \frac{m}{Np} \right)
			\cot^{(j-1)} \left( \frac{\pi n}{Np} \right),
			\\
			\widehat{L}_k^{\chi, \psi, *} (j; x)
			&\coloneqq
			-\frac{N^{k-j-1}}{(2\pi\iu)^j} \frac{(j-1)!}{k-j} p^{k-2j-1}
			\left(
			\vphantom{\sum_{1 \le n \le Np-1}}
			((-1)^j + (-1)^{k-j}) \zeta(j) \chi(0) \psi(0) B_{k-j}
			\right.
			\\
			&\phant
			-(1+(-1)^j) \iu^k \frac{2^{2j - k}}{(Np)^{k-j-1}} \frac{k-j}{j!} B_j \pi^j \psi(0)
			\sum_{1 \le l \le Np-1} \widehat{\chi}(l) 
			\cot^{(k-j-1)} \left( \frac{\pi l}{Np} \right) 
			\\
			&\phant
			+ (1+(-1)^j) \frac{(-1)^{j/2 - 1} 2^j B_j \pi^j}{j!}
			\left( - \chi(0) + \left( \frac{1}{N^{k-j-1}} - 1 \right) \widehat{\chi} (0) \right) \psi(0) B_{k-j}
			&\phant
			\\
			&\phant
			\left.
			+(-1)^{j} B_{k-j} \chi(0)
			\sum_{1 \le n \le Np-1} 
			\left( (-1)^k \psi(n)  + \psi(-n) \right)
			\zeta \left( j; \frac{n}{Np} \right)
			+ \frac{\delta_{j, 1}}{2} \left( \chi(0) - \widehat{\chi}(0) \right) \psi(0).
			\right).
		\end{align}
		Here, $ L_k^{\chi, \psi, *} (j; x) = 0 $ if $ \chi(0) = \psi(0) = 0 $.
		
		\item \label{item:prop:Eisenstein_period_map_L-func:other_rep}
		We have
		\begin{align}
			\widehat{L}_k^{\chi, \psi} (j; x)
			&=
			-\frac{N^{k-1}}{j(k-j)} p^{k-j-1}
			\sum_{1 \le m \le Np-1} 
			\chi(m) 
			B_{k-j} \left( \frac{m}{Np} \right)
			\widetilde{B}_j^{\widehat{\psi}} \left( \frac{mx}{N} \right)
			\\
			&\phant
			+ \frac{(j-1)!}{(2\pi\iu)^j} \frac{N^{k-1-j}}{k-j} p^{k-2j-1}
			(1 + (-1)^k) \psi(0) \zeta(k-1)
			\\
			&\phant
			- \frac{\delta_{j, 1}}{2} \frac{(j-1)!}{(2\pi\iu)^j} \frac{N^{k-1-j} p^{k-1-2j}}{k-j} 
			\sum_{1 \le m \le N-1} 
			\chi(pm) \widehat{\psi}(rm) 
			B_{k-j} \left( \frac{m}{N} \right)
			\\
			&=
			-(-1)^j \iu^{-k} 2^{-k} \frac{1}{p^j}
			\sum_{\substack{
					1 \le n \le Np-1, \\
					p \nmid n
			}} 
			\psi(n)
			\cot_{\widehat{\chi}}^{(k-j-1)} \left( \frac{\pi nx}{N} \right)
			\cot^{(j-1)} \left( \frac{\pi n}{Np} \right)
			\\
			&\phant
			+ \frac{\iu^j N^{k-j-1}}{2^j (k-j)} p^{k-2j-1}
			\chi(0) \left( B_{k-j} + \frac{\delta_{j, 1}}{2} \right)
			\sum_{1 \le n \le Np-1} 
			\psi(n) \cot^{(j-1)} \left( \frac{\pi n}{Np} \right)
			\\
			&\phant 
			+ \left( \frac{1}{(Np)^{k-j-1}} - 1 \right) B_{k-j}
			\sum_{1 \le n \le N-1} 
			\widehat{\chi} (rn) \psi(pn)
			\cot^{(j-1)} \left( \frac{\pi n}{N} \right),
		\end{align}
		where, for a periodic map $ \varphi \colon \Z/N\Z \to \bbC $, we define
		\begin{align}
			\widetilde{B}_j^\varphi (z)
			&\coloneqq
			\sum_{m \in \Z/N\Z} \varphi(m) 
			\widetilde{B}_j \left( z + \frac{m}{N} \right),
			\\
			\cot_{\varphi}^{} (z)
			&\coloneqq
			\sum_{m \in \Z/N\Z} \varphi(m) 
			\cot \left( z + \frac{\pi m}{N} \right).
		\end{align}
		
		In particular, if $ \chi(0) = \psi(0) = 0 $ and there exists $ \varepsilon, \varepsilon' \in \{ \pm 1 \} $ such that 
		$ \chi(-m) = \varepsilon \chi(m) $ and $ \psi(-m) = \varepsilon' \psi(m) $ for any $ m \in \Z/N\Z $, then we have
		\begin{alignat}{2}
			\widehat{L}_{E_k^{\chi, \psi}}(j; x)
			&=
			-\frac{N^{k-1}}{j(k-j)} p^{k-j-1}
			\sum_{1 \le m \le Np-1} 
			\chi(m) 
			B_{k-j} \left( \frac{m}{Np} \right)
			\widetilde{B}_{j}^{\widehat{\psi}} \left( \frac{mx}{N} \right)
			\\
			&=
			-(-1)^j \iu^{-k} 2^{-k} \frac{1}{p^j}
			\sum_{\substack{
					1 \le n \le Np-1, \\
					p \nmid n
			}} 
			\psi(n)
			\cot_{\widehat{\chi}}^{(k-j-1)} \left( \frac{\pi nx}{N} \right)
			\cot^{(j-1)} \left( \frac{\pi n}{Np} \right)
			& & \quad \text{ if } \varepsilon \varepsilon' = (-1)^{j-1}, 
			\\
			S_{E_k^{\chi, \psi}}(x)
			&=
			-\frac{N^{k-1}}{k-1}
			\sum_{1 \le m \le Np-1} \chi(m) 
			B_{1} \left( \frac{m}{Np} \right)
			\widetilde{B}_{k-1}^{\widehat{\psi}} \left( \frac{mx}{N} \right)
			\\
			&=
			\left( \frac{\iu}{2} \right)^k \frac{1}{p^{k-1}}
			\sum_{\substack{
					1 \le n \le Np-1, \\
					p \nmid n
			}} 
			\psi(n)
			\cot_{\widehat{\chi}} \left( \frac{\pi nx}{N} \right)
			\cot^{(k-2)} \left( \frac{\pi n}{Np} \right)
			& & \quad \text{ if } \varepsilon \varepsilon' = (-1)^{k}.
		\end{alignat}
	\end{enumerate}
\end{prop}

\begin{proof}
	\zcref{item:prop:Eisenstein_period_map_L-func:rep} 
	We have
	\begin{align}
		&\phant
		L_{E_k^{\chi, \psi}}(s; x)
		\\
		&=
		N^s
		\sum_{n=1}^\infty \frac{\sigma_{k-1}^{\chi, \psi} (n) + (-1)^k \sigma_{k-1}^{\chi^-, \psi^-} (n)}{n^s} \bm{e} \left( \frac{nx}{N} \right)
		\\
		&=
		N^s
		\sum_{d, m=1}^\infty \frac{d^{k-1} ( \chi(m) \psi(d) + (-1)^k \chi(-m) \psi(-d) )}{(dm)^s} \bm{e} \left( \frac{dmx}{N} \right)
		\\
		&=
		N^{-s + k -1} p^{-2s + k -1} \sum_{1 \le m, n \le Np} 
		\left( \chi(m) \psi(n) + (-1)^k \chi(-m) \psi(-n) \right)
		\bm{e} \left( \frac{mnx}{N} \right) 
		\zeta \left( s-k+1; \frac{m}{Np} \right) \zeta \left( s; \frac{n}{Np} \right).
	\end{align}
	
	\zcref{item:prop:Eisenstein_period_map_L-func:special_value}
	By \zcref{item:prop:Eisenstein_period_map_L-func:rep,lem:Hurwitz_zeta_special_value}, we have 
	\begin{align}
		&\phant
		\widehat{L}_{E_k^{\chi, \psi}}(j; x)
		=
		-\frac{(j-1)!}{(2\pi\iu)^j} L_{E_k^{\chi, \psi}}(j; x)
		\\
		&=
		\frac{N^{k-1-j} p^{k-1-2j}}{(2\pi\iu)^j} \frac{(j-1)!}{k-j} 
		\sum_{1 \le m, n \le Np} 
		\left( \chi(m) \psi(n) + (-1)^k \chi(-m) \psi(-n) \right)
		\bm{e} \left( \frac{mnx}{N} \right) 
		B_{k-j} \left( \frac{m}{Np} \right) \zeta \left( j; \frac{n}{Np} \right).
	\end{align}
	We apply a transformation
	\[
	m \mapsto
	\begin{cases}
		Np - m & \text{ if } 1 \le m \le Np - 1, \\
		Np & \text{ if } m = Np, 
	\end{cases}
	\quad
	n \mapsto
	\begin{cases}
		Np - n & \text{ if } 1 \le n \le Np - 1, \\
		Np & \text{ if } n = Np.
	\end{cases}
	\]
	By this transformation, we have
	\begin{alignat}{2}
		\chi(-m) \psi(-n) &\mapsto \chi(m) \psi(n),
		& \quad
		\bm{e} \left( -\frac{mnx}{N} \right) &\mapsto \bm{e} \left( -\frac{mnx}{N} \right),
		\\
		B_{k-j} \left( \frac{m}{Np} \right) &\mapsto
		(-1)^{(k-j)(\delta_{m, Np} + 1)} B_{k-j} \left( \frac{m}{Np} \right),
		& \quad
		\zeta \left( j; \frac{n}{Np} \right)
		&\mapsto
		\zeta \left( j; 1 - \frac{n}{Np} + \delta_{n, Np} \right).
	\end{alignat}
	Thus, we obtain
	\[
	L_{E_k^{\chi, \psi}}(j; x)
	=
	\frac{N^{k-1-j} p^{k-1-2j}}{(2\pi\iu)^j} \frac{(j-1)!}{k-j} \sum_{1 \le m, n \le Np} 
	\chi(m) \psi(n) \bm{e} \left( \frac{mnx}{N} \right)
	B_{k-j} \left( \frac{m}{Np} \right) Z_{m, n},
	\]
	where
	\[
	Z_{m, n} \coloneqq
	 \zeta \left( j; \frac{n}{Np} \right)
	 + (-1)^{k + (k-j)(\delta_{m, Np} + 1)} \zeta \left( j; 1 - \frac{n}{Np} + \delta_{n, Np} \right).
	\]
	For $ e, e' \in \{ 0, 1 \} $, let
	\[
	L_{e, e'} \coloneqq
	\sum_{m \in S_e} \sum_{n \in S_{e'}} 
	\chi(m) \psi(n) \bm{e} \left( \frac{mnx}{N} \right)
	B_{k-j} \left( \frac{m}{Np} \right) Z_{m, n},
	\]
	where $ S_0 \coloneqq \{ Np \} $ and $ S_1 \coloneqq \{ 1, \dots, Np-1 \} $.
	
	Since $ Z_{Np, Np} = (1 + (-1)^k) \zeta(j) $, we have
	\[
	L_{0, 0} =
	((-1)^j + (-1)^{k-j}) \zeta(j) \chi(0) \psi(0) B_{k-j}.
	\]
	
	For $ 1 \le m \le Np-1 $, we have
	\[
	Z_{m, Np} =
	\zeta(j) + (-1)^j \zeta(j)
	=
	(1+(-1)^j) \frac{(-1)^{j/2 - 1} 2^j B_j \pi^j}{j!}.
	\]
	By \zcref{lem:periodic_Bernoullli_poly_disc_Fourier_exp}, we have
	\begin{align}
		&\phant
		\sum_{1 \le m \le Np-1} \chi(m) B_{k-j} \left( \frac{m}{Np} \right)
		\\
		&=
		\sum_{1 \le m \le Np-1} \sum_{l \in \Z/Np\Z} 
		\widehat{\chi}(l) B_{k-j} \left( \frac{m}{Np} \right) 
		\bm{e} \left( \frac{lm}{Np} \right)
		\\
		&=
		\sum_{1 \le l \le Np-1} \widehat{\chi}(l) 
		\left(
			\frac{k-j}{(2\iu)^{k-j} (Np)^{k-j-1}}
			\cot^{(k-j-1)} \left( \frac{\pi l}{Np} \right) 
			- B_{k-j} - \frac{\delta_{j, 1}}{2}
		\right)
		+ \left( \frac{1}{N^{k-j-1}} - 1 \right) \widehat{\chi} (0) B_{k-j}
		\\
		&=
		\frac{k-j}{(2\iu)^{k-j} (Np)^{k-j-1}}
		\sum_{1 \le l \le Np-1} \widehat{\chi}(l) 
		\cot^{(k-j-1)} \left( \frac{\pi l}{Np} \right) 
		+ \left( - \chi(0) + \left( \frac{1}{N^{k-j-1}} - 1 \right) \widehat{\chi} (0) \right) B_{k-j}
		\\
		&\phant
		+ \frac{\delta_{j, 1}}{2} \left( \widehat{\chi}(0) - \chi(0) \right).
	\end{align}
	Thus, we have
	\begin{align}
		L_{1, 0} &=
		-(1+(-1)^j) \iu^k \frac{2^{2j - k}}{(Np)^{k-j-1}} \frac{k-j}{j!} B_j \pi^j \psi(0)
		\sum_{1 \le l \le Np-1} \widehat{\chi}(l) 
		\cot^{(k-j-1)} \left( \frac{\pi l}{Np} \right) 
		\\
		&\phant
		+ (1+(-1)^j) \frac{(-1)^{j/2 - 1} 2^j B_j \pi^j}{j!}
		\left( - \chi(0) + \left( \frac{1}{N^{k-j-1}} - 1 \right) \widehat{\chi} (0) \right) \psi(0) B_{k-j}
		+ \frac{\delta_{j, 1}}{2} \left( \widehat{\chi}(0) - \chi(0) \right) \psi(0).
	\end{align}
	
	For $ 1 \le n \le Np-1 $, we have
	\[
	Z_{Np, n} =
	\zeta \left( j; \frac{n}{Np} \right) + (-1)^k \zeta \left( j; 1 - \frac{n}{Np} \right).
	\]
	Thus, we have
	\begin{align}
		L_{0, 1} &=
		(-1)^{k-j} B_{k-j} \chi(0)
		\sum_{1 \le n \le Np-1} \psi(n) 
		\left( \zeta \left( j; \frac{n}{Np} \right) + (-1)^k \zeta \left( j; 1 - \frac{n}{Np} \right) \right)
		\\
		&=
		(-1)^{j} B_{k-j} \chi(0)
		\sum_{1 \le n \le Np-1} 
		\left( (-1)^k \psi(n)  + \psi(-n) \right)
		\zeta \left( j; \frac{n}{Np} \right).
	\end{align}
	
	For $ 1 \le m, n \le Np - 1 $, we have
	\begin{align}
		Z_{m, n} &=
		\zeta \left( j; \frac{n}{Np} \right) + (-1)^{j} \zeta \left( j; 1 - \frac{n}{Np} \right)
		=
		\sum_{l \in \Z} \frac{1}{(l + n/Np)^j}
		=
		\frac{(-1)^{j+1} \pi^j}{(j-1)!}
		\cot^{(j-1)} \left( \frac{\pi n}{Np} \right).
	\end{align}
	Thus, we obtain the claim.
	
	\zcref{item:prop:Eisenstein_period_map_L-func:other_rep}
	By \zcref{lem:periodic_Bernoullli_poly_Fourier_exp}, we have
	\begin{align}
		&\phant
		\sum_{1 \le n \le Np-1} \psi(n) \bm{e} \left( \frac{mnx}{N} \right) Z_{m, n}
		\\
		&=
		\sum_{1 \le n \le Np-1} \psi(n) \bm{e} \left( \frac{mnx}{N} \right)
		\sum_{l \in \Z} \frac{1}{(l + n/Np)^j}
		\\
		&=
		\sum_{l \in \Z \smallsetminus \{ 0 \}} 
		\left(
			\frac{\psi(l)}{(l/Np)^j} \bm{e} \left( \frac{mlx}{N} \right)
			- \frac{\psi(0)}{l^j}
		\right)
		\\
		&=
		(Np)^j \sum_{l \in \Z \smallsetminus \{ 0 \}} \sum_{n' \in \Z/N\Z}
		\frac{\widehat{\psi}(n')}{l^j} \bm{e} \left( \frac{ln'}{N} + \frac{mlx}{N} \right)
		- (1 + (-1)^j) \psi(0) \zeta(j)
		\\
		&=
		-\frac{(2\pi\iu Np)^j}{j!} \sum_{n' \in \Z/N\Z} \widehat{\psi}(n')
		\widetilde{B}_j \left( \frac{n' + mx}{N} \right)
		- (1 + (-1)^j) \psi(0) \zeta(j)
		- \frac{\delta_{j, 1}}{2} \bm{1}_{p\Z} (m) \widehat{\psi} (mx).
	\end{align}
	Thus, we obtain the first equality.
	
	By \zcref{lem:periodic_Bernoullli_poly_disc_Fourier_exp}, we have
	\begin{align}
		&\phant
		\sum_{1 \le m \le Np-1} 
		\chi(m) \bm{e} \left( \frac{mnx}{N} \right)
		B_{k-j} \left( \frac{m}{Np} \right)
		\\
		&=
		\sum_{1 \le m \le Np-1} \sum_{m' \in \Z/N\Z}
		\widehat{\chi}(m')
		B_{k-j} \left( \frac{m}{Np} \right)
		\bm{e} \left( \frac{m (m' + nx)}{N} \right)
		\\
		&=
		\sum_{\substack{
				m' \in \Z/N\Z, \\
				m' - nx \notin N \Z
		}}
		\widehat{\chi}(m')
		\left(
			\frac{k-j}{(2\iu)^{k-j} (Np)^{k-j-1}} \cot^{(k-j-1)} \left( \frac{\pi (m' + nx)}{N} \right)
			- B_{k-j} - \frac{\delta_{j, 1}}{2} 
		\right)
		\\
		&\phant
		+ \bm{1}_{p\Z} (n) \widehat{\chi} (nx) \left( \frac{1}{(Np)^{k-j-1}} - 1 \right) B_{k-j} 
		\\
		&=
		\frac{k-j}{(2\iu)^{k-j} (Np)^{k-j-1}}
		\sum_{\substack{
				m' \in \Z/N\Z, \\
				m' - nx \notin N \Z
		}}
		\widehat{\chi}(m')
		\cot^{(k-j-1)} \left( \frac{\pi  (m' - nx)}{N} \right)
		\\
		&\phant
		- B_{k-j} \chi(0)
		+ \bm{1}_{p\Z} (n) \widehat{\chi} (nx) \left( \frac{1}{(Np)^{k-j-1}} - 1 \right) B_{k-j}
		- \frac{\delta_{j, 1}}{2} \chi(0).
	\end{align}
	Thus, we obtain the second equality.
\end{proof}

By \zcref{thm:reciprocity_gen'd_Ded_sum}, we obtain the following reciprocity formula.

\begin{cor}\label{cor:reciprocity_gen'd_Ded_sum_Gamma(N)}
	For any periodic maps $ \chi, \psi \colon \Z/N\Z \to \bbC $ and any $ \gamma = \smat{a & b \\ c & d} \in \Gamma(N) $ with $ cx+d \neq 0 $, the generalized Dedekind sum $ S_{E_k^{\chi, \psi}}(x) $ computed in \zcref{prop:Eisenstein_period_map_L-func} \zcref{item:prop:Eisenstein_period_map_L-func:special_value} satisfies the modular transformation 
	\[
		\left( \eval{S_{E_k^{\chi, \psi}}}_{2-k} \gamma \right) (x) - S_{E_k^{\chi, \psi}} (x)
		=
		R_{E_k^{\chi, \psi}, k, \gamma} (x)  - \frac{a_{E_k^{\chi, \psi}, 0}^{(x)}}{p^k} \frac{c}{cx+d},
	\]
	where
	\begin{align}
		R_{E_k^{\chi, \psi}, \gamma} (X) 
		&\coloneqq
		-\sum_{j=1}^{k-1} \binom{k-2}{j-1} \widehat{L}_{E_k^{\chi, \psi}} \left( j; -\frac{d}{c} \right) \left( \frac{cX+d}{c} \right)^{k-j-1}
		\\
		&\phant
		- \frac{a_{E_k^{\chi, \psi}, 0}}{k-1} \left( \left( \frac{cX+d}{c} \right)^{k-1} + \frac{1}{(-c)^{k-1} (cX+d)} \right)
		\quad \in \frac{1}{cX+d} \bbC[X]
	\end{align}
	and $ \widehat{L}_{E_k^{\chi, \psi}} \left( j; -d/c \right) $, $ a_{E_k^{\chi, \psi}, 0} $, and $ a_{E_k^{\chi, \psi}, 0}^{(x)} $ are computed in \zcref{prop:Eisenstein_period_map_L-func} \zcref{item:prop:Eisenstein_period_map_L-func:special_value} and \zcref{prop:Eisenstein_period_map}.
\end{cor}

\zcref{cor:main_gen'd_Ded_sum_Eisenstein} is the special case of this corollary.

\section{Higher weight generalized Dedekind sums of level 2} \label{sec:Dedekind_sum_level_2}


Throughout this section, we fix an even integer $ g \ge 0 $. 
As before, we fix a rational number $ x = r/p \in \Q $ with coprime integers $ r \in \Z $ and $ p \in \Z_{>0} $.

In this section, we focus on the following trigonometric sums.
Their relation to the signature in TQFT will be established later.

\begin{dfn}
	We define the \emph{generalized Dedekind sums of level $ 2 $ and weight $ -g $} as
	\[
	S_g^{\odd} \left( x \right)
	\coloneqq
	\frac{1}{p^{g+1}} 
	\sum_{\substack{
			1 \le n \le p - 2, \\
			\text{$ n $ odd}
	}}
	\frac{\cot^{(g)} (\pi n/2p)}{\sin (\pi nx)} 
	\in \Q.
	\]
\end{dfn}

\begin{rem}
	The sum $ S_g^{\odd} (x) $ defines an odd map $ \Q/2\Z \to \Q $.
\end{rem}

\begin{rem} \label{rem:Dedekind_sum_level_2_express}
	By taking $ n \mapsto 2p - n $, we have
	\[
	S_g^{\odd} \left( x \right)
	=
	\frac{1}{2p^{g+1}} 
	\sum_{\substack{
			1 \le n \le 2p, \\
			\text{$ n $ odd}, \, n \neq p
	}}
	\frac{\cot^{(g)} (\pi n/2p)}{\sin (\pi nx)}.
	\]
\end{rem}

In order to describe quantum modularity of $ S_g^{\odd} (x) $, we need the following Eisenstein series of level 2.

\begin{dfn} \label{dfn:Eisenstein}
	For an even integer $ k \in 2\Z $, we define 
	\[
	E_k^{\odd} (\tau)
	\coloneqq
	\sum_{n \ge 1, \, \odd} \sigma_{k - 1} (n) q^{n/2},
	\quad
	\sigma_{k - 1} (n) \coloneqq \sum_{d \mid n} d^{k - 1}.
	\]
\end{dfn}

The sum $ S_g^{\odd} (x) $ satisfies the following properties.
Here, we recall that the congruence subgroup $ \Gamma(2) $ is generated by $ \smat{1 & 2 \\ 0 & 1} $ and $ \smat{1 & 0 \\ 2 & 1} $.

\begin{thm} \label{thm:Dedekind_sum_level_2}
	\begin{enumerate}
		\item \label{item:thm:Dedekind_sum_level_2:expression}
		We have
		\begin{align}
			S_g^{\odd} \left( x \right)
			&=
			(-1)^{g/2 + 1} 2^{g+2} S_{E_{g+2}^{\odd}} (x)
			\\
			&=
			\frac{\iu}{p^{g+1}}
			\sum_{\substack{
					1 \le m, n < 2p, \\
					\text{$ m, n $ odd}
			}}
			\bm{e} \left( \frac{mnx}{2} \right)
			B_1 \left( \frac{m}{2p} \right)
			\cot^{(g)} \left( \frac{\pi n}{2p} \right).
			\\
			&=
			\frac{(-1)^{g/2} 2^{2g+1}}{g+1}
			\sum_{\substack{
					1 \le m \le 2p, \\
					m \, \odd
			}}
			\left(
			\widetilde{B}_{g+1} \left( \frac{mx}{2} \right) - \widetilde{B}_{g+1} \left( \frac{mx}{2} + \frac{1}{2} \right)
			\right)
			B_1 \left( \frac{m}{2p} \right).
		\end{align}
		
		\item \label{item:thm:Dedekind_sum_level_2:asymp}
		For any $ \varepsilon > 0 $, uniformly for $ \tau $ with $ \varepsilon \le \arg(\tau) \le \pi - \varepsilon $, we have
		\begin{align}
			E_{-g}^{\odd} \left( x + \tau \right)
			&=
			- \frac{(\pi \iu)^{g+1}}{(g+1)!}
			\frac{a_{E_{g+2}^{\odd}, 0}^{(x)}}{p^{g+2}} \frac{1}{\tau}
			+ \iu \frac{\pi^{g+1}}{2^{g+2} g!} S_g^{\odd} \left( x \right)
			+ O(\tau) \quad \text{ as } \tau \to 0,
		\end{align}
		where 
		\[
		a_{E_{g+2}^{\odd}, 0}^{(x)} 
		=
		\bm{1}_2 (p) (-1)^{r+1} \frac{2^{g+2} - 1}{4(g+2)} B_{g+2}
		\]
		
		\item \label{item:thm:Dedekind_sum_level_2:quantum_modular}
		For any $ \gamma = \smat{a & b \\ c & d} \in \Gamma(2) $, we have
		\[
		(cx + d)^{g} S_g^{\odd} (\gamma x) - S_g^{\odd} (x)
		=
		(-1)^{g/2 + 1} 2^{g+2} R_{E_{g+2}^{\odd}, \gamma} (x)
		+ (-1)^{g/2} 2^{g+2} a_{E_{g+2}^{\odd}, 0}^{(x)} \frac{c}{p^{g+1}(cr+dp)}.
		\]
		where 
		\[
		R_{E_{g+2}^{\odd}, \gamma} (X)
		=
		\int_{\gamma^{-1} (\iu \infty)}^{\iu \infty} E_{g+2}^{\odd} (z) (X - z)^{g} dz
		\in \bbC[X].
		\]
		In particular, in the case when $ g=0, 2 $, we have
		\begin{align}
			S_0^{\odd} \left( \frac{x}{2x+1} \right) - S_0^{\odd} (x)
			&=
			\frac{1}{2} + \frac{\bm{1}_2 (p) (-1)^{r+1}}{2p(2r+p)},
			\\
			(2x + 1)^2 S_2^{\odd} \left( \frac{x}{2x+1} \right) - S_2^{\odd} (x)
			&=
			2x^2 + 2x + 1 + \frac{\bm{1}_2 (p) (-1)^{r+1}}{p^3 (2r+p)}.
		\end{align}
	\end{enumerate}
\end{thm}

\begin{figure}[htbp]
	\centering
	
	\begin{subfigure}{0.45\linewidth}
		\centering
		\includegraphics[width=\linewidth]{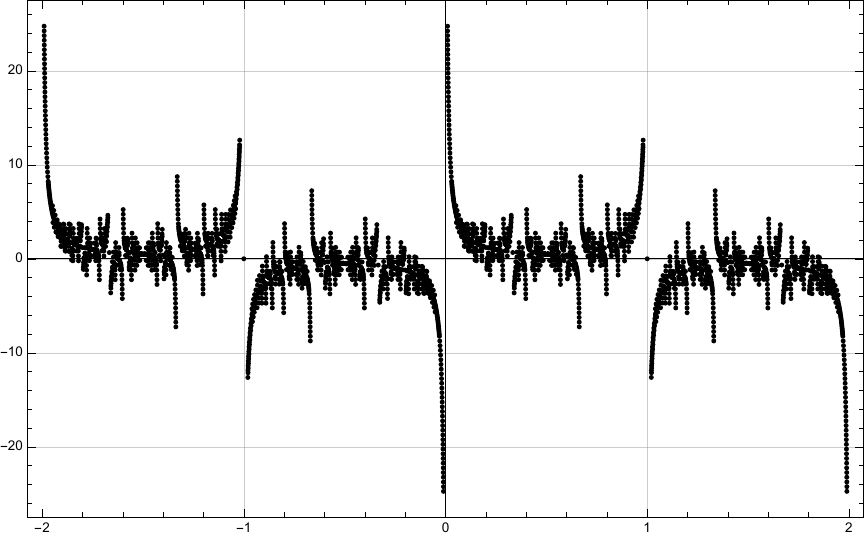}
		\caption{The graph of $ S_0^{\odd} (x) $.}
	\end{subfigure}
	\hfill
	\begin{subfigure}{0.45\linewidth}
		\centering
		\includegraphics[width=\linewidth]{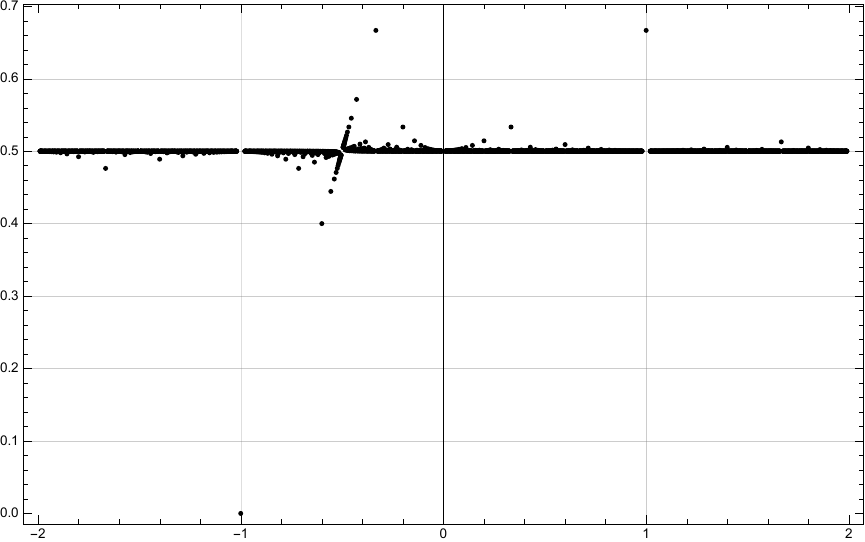}
		\caption{The graph of $ S_0^{\odd} (x/(2x+1)) - S_0^{\odd} (x) $.}
	\end{subfigure}
	
	\medskip
	
	\begin{subfigure}{0.45\linewidth}
		\centering
		\includegraphics[width=\linewidth]{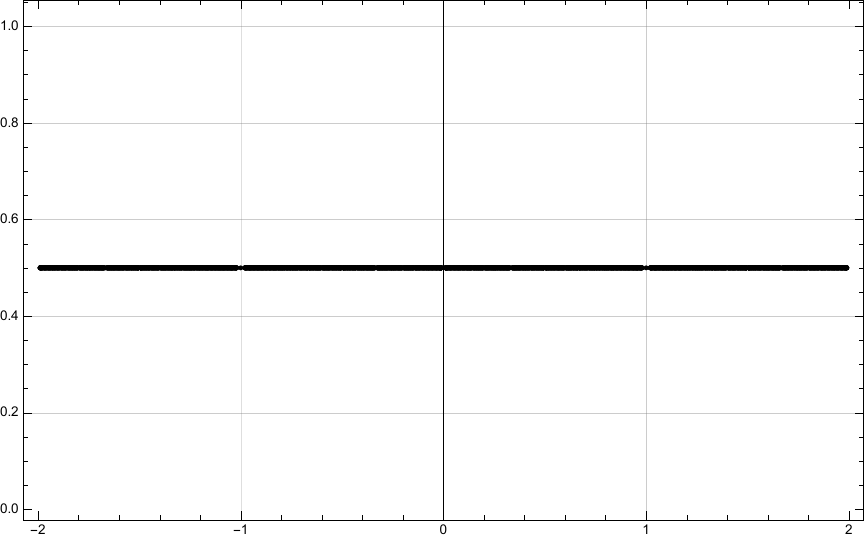}
		\captionsetup{width=1.3\textwidth}
		\caption{The graph of $ S_0^{\odd} (x/(2x+1)) - S_0^{\odd} (x) - 1/2p (2r+p) $, which coincides with the graph of $ 1/2 $.}
	\end{subfigure}
	
	\caption{Quantum modularity for $ S_0^{\odd} (x) $.}
	\label{fig:S_0}
\end{figure}

\begin{figure}[htbp]
	\centering
	
	\begin{subfigure}{0.45\linewidth}
		\centering
		\includegraphics[width=\linewidth]{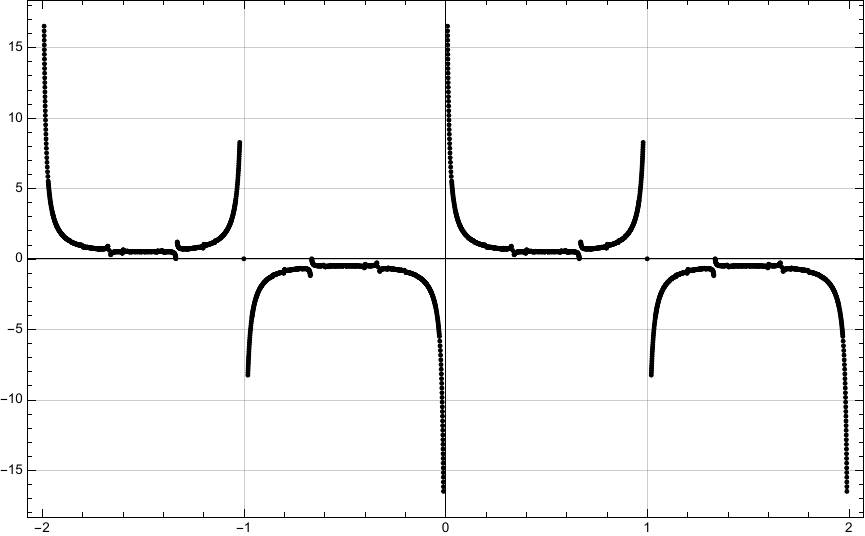}
		\caption{The graph of $ S_2^{\odd} (x) $.}
	\end{subfigure}
	\hfill
	\begin{subfigure}{0.45\linewidth}
		\centering
		\includegraphics[width=\linewidth]{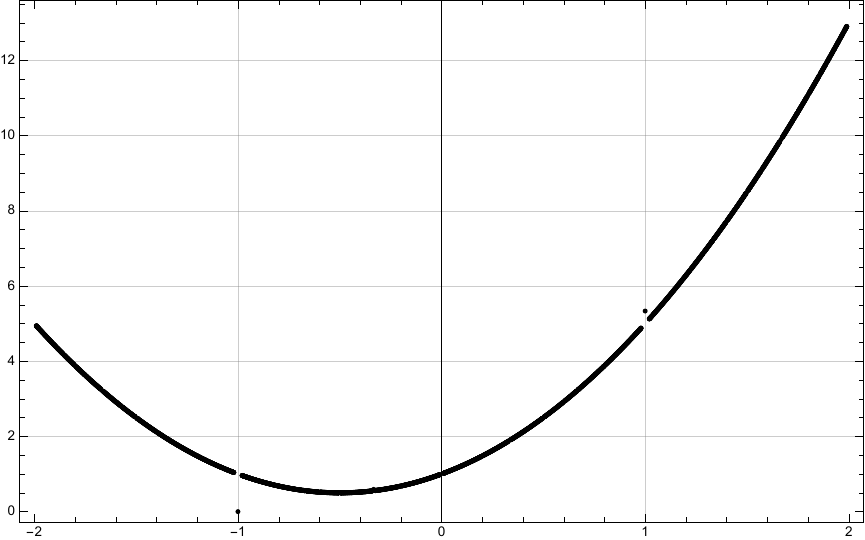}
		\caption{The graph of $ (2x+1)^2 S_2^{\odd} (x/(2x+1)) - S_2^{\odd} (x) $.}
	\end{subfigure}
	
	\medskip
	
	\begin{subfigure}{0.45\linewidth}
		\centering
		\includegraphics[width=\linewidth]{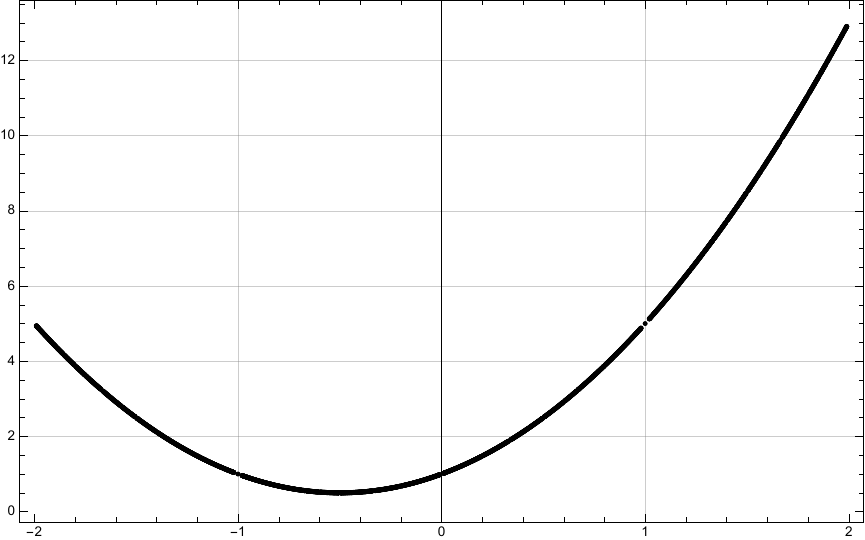}
		\captionsetup{width=1.2\textwidth}
		\caption{The graph of $ (2x+1)^2 S_2^{\odd} (x/(2x+1)) - S_2^{\odd} (x) - 1/p^3 (2r+p) $, which coincides with the graph of $ 2x^2 + 2x + 1 $.}
	\end{subfigure}
	
	\caption{Quantum modularity for $ S_2^{\odd} (x) $.}
	\label{fig:S_2}
\end{figure}

The quantum modularity of $ S_0^{\odd} (x) $ and $ S_2^{\odd} (x) $ is illustrated in \zcref{fig:S_0,fig:S_2}.
Here, the graphs are plotted for $ x = r/p \in (-2, 2) $ with coprime odd integers $ r $ and $1 \le p \le 99$.

\begin{proof}
	By \zcref{prop:Eisenstein_period_map}, we have $ E_{g+2}^{\odd} (\tau) = \frac{1}{2} E_{g+2}^{\bm{1}_2, \bm{1}_2} (\tau) $, where $ \bm{1}_2 \coloneqq \bm{1}_{1 + 2\Z} $ is the characteristic function of $ 1 + 2\Z $.
	In particular, the Eisenstein series $ E_{g+2}^{\odd} (\tau) $ is a modular form for $ \Gamma(2) $ of weight $ g+2 $.
	Thus, \zcref{item:thm:Dedekind_sum_level_2:expression} follows from \zcref{prop:Eisenstein_period_map} and
	\[
	\widehat{\bm{1}_2} (l) 
	= \frac{(-1)^l}{2},
	\quad
	\widetilde{B}_j^{\widehat{\bm{1}_2}} (\alpha) 
	= \widetilde{B}_j (\alpha) - \widetilde{B}_j \left( \alpha + \frac{1}{2} \right),
	\quad
	\cot_{\widehat{\bm{1}_2}} (z)
	=
	\frac{1}{2} \left( \cot z + \tan z \right)
	=
	\frac{1}{\sin (2z)}.
	\]

	By \zcref{rem:Eichler_int_Fourier}, we have
	\[
	E_{-g}^{\odd} (\tau) = 
	-\frac{(\pi\iu)^{g+1}}{g!} \calE_{E_{g+2}^{\odd}} (\tau).
	\]
	By \zcref{prop:Eisenstein_period_map} \zcref{item:prop:Eisenstein_period_map:cusp}, we also have
	\begin{align}
		a_{E_{g+2}^{\odd}, 0}^{(x)} 
		&=
		-\frac{2^{g+1}}{2(g+2)}
		\sum_{l, m \in \Z/2\Z} \bm{1}_2 (-pl) \widehat{\bm{1}_2 } (rl)
		\bm{e} \left( -\frac{lm}{2} \right) \widetilde{B}_{g+2} \left( \frac{m}{2} \right)
		\\
		&=
		\bm{1}_2 (p) (-1)^{r+1} \frac{2^{g-1}}{g+2}
		\sum_{m \in \Z/2\Z} (-1)^m \widetilde{B}_{g+2} \left( \frac{m}{2} \right)
		\\
		&=
		\bm{1}_2 (p) (-1)^{r+1} \frac{2^{g+2} - 1}{4(g+2)} B_{g+2}
	\end{align}
	since
	\[
	B_j(0) = B_j,
	\quad
	B_j \left( \frac{1}{2} \right)
	=
	(2^{-j} - 1) B_j.
	\]
	Thus, \zcref{item:thm:Dedekind_sum_level_2:asymp} follows from \zcref{prop:Eichler_int_Ded_sum}.
	
	The first part of \zcref{item:thm:Dedekind_sum_level_2:quantum_modular} follows from \zcref{thm:reciprocity_gen'd_Ded_sum}.
	The second part follows from the following lemma.
\end{proof}

\begin{lem} \label{lem:Eisenstein_wt_0_2}
	We have
	\[
	R_{E_2^{\odd}, \gamma} (X)
	=
	-\frac{1}{8},
	\quad
	R_{E_4^{\odd}, \gamma} (X)
	=
	\frac{1}{16} (2X^2 + 2X + 1).
	\]
	Equivalently,
	\begin{align}
		E_0^{\odd} \left( \frac{\tau}{2\tau + 1} \right) - E_0^{\odd} (\tau)
		&=
		\frac{\pi \iu}{8},
		\\
		(2\tau + 1)^2 E_{-2}^{\odd} \left( \frac{\tau}{2\tau + 1} \right) - E_{-2}^{\odd} (\tau)
		&=
		-\frac{(\pi \iu)^3}{32} (2\tau^2 + 2\tau + 1).
	\end{align}
\end{lem}

\begin{proof}
	Although the period polynomial $ R_{E_{k}^{\odd}, \gamma} (\tau) $ can be computed from \zcref{prop:Eisenstein_period_map_L-func}, we instead present here a derivation using the modular transformation formula of $ E_{2-k}^{\odd} (\tau) $.
	
	The weight 2 case is proved by March\'{e}--Masbaum~\cite[Proposition 6.10]{Marche-Masbaum}.
	
	We prove the weight 0 case.
	We recall that Jacobi's theta functions are defined as
	\begin{align}
		\theta_2 (\tau)
		&\coloneqq
		q^{1/8} \prod_{n=1}^\infty (1 - q^n) (1 + q^n)^2,
		\\
		\theta_3 (\tau)
		&\coloneqq
		\prod_{n=1}^\infty (1 - q^n) (1 + q^{n - 1/2})^2,
		\\
		\theta_4 (\tau)
		&\coloneqq
		\prod_{n=1}^\infty (1 - q^n) (1 - q^{n - 1/2})^2
	\end{align}
	and these functions satisfy modular transformation formulas
	\[
	\pmat{\theta_2 \\ \theta_3 \\ \theta_4} (\tau)
	= 
	\pmat{\bm{e} (1/8) \theta_2 \\ \theta_3 \\ \theta_4} ( \tau),
	\quad
	\pmat{\theta_2 \\ \theta_3 \\ \theta_4} \left( -\frac{1}{\tau} \right)
	= 
	\sqrt{\frac{\tau}{\iu}}
	\pmat{\theta_4 \\ \theta_3 \\ \theta_2} (\tau).
	\]
	Since we have 
	\[
	\frac{1}{4} \log \frac{\theta_3 (\tau)}{\theta_4 (\tau)}
	=
	\frac{1}{2} \log \sum_{n \ge 1, \, \odd} \frac{1 + q^{n/2}}{1 - q^{n/2}}
	=
	E_0^{\odd} (\tau)
	\quad \text{ and } \quad
	\pmat{1 & 0 \\ 2 & 1}
	=
	\pmat{0 & -1 \\ 1 & 0} \pmat{1 & -2 \\ 0 & 1} \pmat{0 & 1 \\ -1 & 0},
	\]
	we obtain the claim.
\end{proof}

\begin{lem} \label{lem:Dedekind_sum_level_2:h=1}
	It holds that
	\begin{align}
		S_0^{\odd} \left( \frac{1}{p} \right)
		&=
		\frac{1}{4} \left( p - \frac{1}{p} \right).
	\end{align}
\end{lem}

\begin{proof}
	By  \zcref{rem:Dedekind_sum_level_2_express} and the double angle formula $ \sin (2x) = 2 \sin x \cos x $, we have
	\begin{align}
		S_0^{\odd} \left( \frac{1}{p} \right)
		&=
		\frac{1}{4p} 
		\sum_{\substack{
				1 \le n \le 2p, \\
				\text{$ n $ odd}, \, n \neq p
		}}
		\frac{1}{\sin^2 (\pi n/2p)}
		\\
		&=
		\frac{1}{4p} 
		\left(
		\sum_{1 \le n \le 2p-1} \frac{1}{\sin^2 (\pi n/2p)}
		- \sum_{1 \le n \le p-1} \frac{1}{\sin^2 (\pi n/p)}
		- 1
		\right).
	\end{align}
	Since $ 1/\sin^2 z = -\cot' z $, by \zcref{lem:periodic_Bernoullli_poly_disc_Fourier_exp} we have
	\begin{equation} \label{eq:sin_sum}
	\sum_{1 \le n \le p-1} \frac{1}{\sin^2 (\pi n/p)}
	=
	-\frac{(-2\iu p)^2}{2} \left( \widetilde{B}_2 \left( \frac{0}{p} \right) - \frac{B_2}{p^2} \right)
	=
	\frac{p^2 - 1}{3}.
	\end{equation}
	Thus, we obtain the formula.
	
\end{proof}


\section{Proof of main theorem} \label{sec:proof}


In this section, we prove \zcref{thm:main,thm:main_expression_sigma}.

As before, we fix a rational number $ x = r/p \in \Q $ with coprime integers $ r \in \Z $ and $ p \in \Z_{>0} $.
Throughout this section, we assume that $ 1 \le r < p $ and $ r $ and $ p $ are coprime.
In this case, one can consider the signatures of $ \SU(2) $-TQFT introduced by Deroin--March\'{e}~\cite{Deroin-Marche}.
We denote by $ \sigma_2 (x) $ the signatures of $ \SU(2) $-TQFT for a closed surface of genus 2.

First, we give a simple expression of $ \sigma_2(x) $.

\begin{prop} \label{prop:sign_TQFT_simple_expression_g=2}
	We have
	\[
	\sigma_2 \left( x \right)
	=
	\frac{2}{p} 
	\sum_{\substack{
			1 \le n \le p - 2, \\
			\text{$ n $ odd}
	}}
	\frac{\cot^3 (\pi n/2p)}{\sin (\pi nx)}
	=
	p^2 S_2^{\odd} \left( x \right) - 2 S_0^{\odd} \left( x \right).
	\]
\end{prop}

\begin{proof}
	The second equality follows from 
	$ \cot'' z  = 2 \cot^3 z + 2 \cot z $.
	
	We prove the first equality.
	 March\'{e}--Masbaum~\cite[Theorem 3.1]{Marche-Masbaum} proved the expression
	 \[
	 \sigma_2 \left( x \right)
	 =
	 \frac{1}{6p^2} - \frac{1}{6} + \frac{1}{4p^2} 
	 \sum_{\substack{
	 		1 \le n \le p - 2, \\
	 		\text{$ n $ odd}
	 }}
	 \frac{T(n; x)}{\sin^3 (\pi n/2p) \sin^2 (\pi nx/2)},
	 \]
	 where
	 \[
	 T(n; x)
	 \coloneqq
	 \sum_{\varepsilon \in \{ \pm 1 \}} (p + \varepsilon)
	 \left(
	 \sin \left( \frac{\pi (2r - 3 \varepsilon)n}{2p} \right)
	 + 3 \sin \left( \frac{\pi (2r + \varepsilon)n}{2p} \right)
	 \right).
	 \]
	By the addition theorem and the triple-angle formula, we have
	\begin{align}
		T(n; r/p)
		&=
		\sum_{\varepsilon \in \{ \pm 1 \}} (p + \varepsilon)
		\left(
		\sin \left( \pi nx \right) \cos \left( \frac{3\pi n}{2p} \right)
		- \varepsilon \cos \left( \pi nx \right) \sin \left( \frac{3\pi n}{2p} \right)
		\right.
		\\
		&\phantom{
			=
			\sum_{\varepsilon \in \{ \pm 1 \}} (p + \varepsilon)
			\left( \right.
		}
		\left.
		+ 3 \sin \left( \pi nx \right) \cos \left( \frac{\pi n}{2p} \right)
		+ 3 \varepsilon \cos \left( \pi nx \right) \sin \left( \frac{\pi n}{2p} \right)
		\right)
		\\
		&=
		\sum_{\varepsilon \in \{ \pm 1 \}} 4(p + \varepsilon)
		\left(
		\sin \left( \pi nx \right) \cos^3 \left( \frac{\pi n}{2p} \right)
		+ \varepsilon \cos \left( \pi nx \right) \sin^3 \left( \frac{\pi n}{2p} \right)
		\right).
	\end{align}
	Thus, we have
	\begin{align}
		\sigma_2 \left( x \right)
		- \frac{1}{6p^2} + \frac{1}{6}
		&=
		\frac{1}{p^2} 
		\sum_{\substack{
				1 \le n \le p - 2, \\
				\text{$ n $ odd}
		}}
		\sum_{\varepsilon \in \{ \pm 1 \}} (p + \varepsilon)
		\left(
		\frac{\cot^3 (\pi n/2p)}{\sin (\pi nx)} + \varepsilon \frac{\cos (\pi nx)}{\sin^2 (\pi nx)}
		\right)
		\\
		&=
		2
		\sum_{\substack{
				1 \le n \le p - 2, \\
				\text{$ n $ odd}
		}}
		\left(
		\frac{1}{p} \frac{\cot^3 (\pi n/2p)}{\sin (\pi nx)} 
		+ \frac{1}{p^2} \frac{\cos (\pi nx)}{\sin^2 (\pi nx)}
		\right).
	\end{align}
	Therefore, the claim follows from 
	\begin{equation} \label{eq:sign_TQFT_simple_expression:tirg_sum}
		S \coloneqq
		\sum_{\substack{
				1 \le n \le p - 2, \\
				\text{$ n $ odd}
		}}
		\frac{\cos (\pi nx)}{\sin^2 (\pi nx)}
		=
		\frac{p^2 - 1}{12}.
	\end{equation}
	
	We prove this equality.
	Applying the substitution $ n \mapsto p-n $, we have
	\[
	S =
	-\sum_{\substack{
			1 \le n \le p - 2, \\
			\text{$ n $ even}
	}}
	\frac{\cos (\pi nx)}{\sin^2 (\pi nx)}
	=
	-\sum_{1 \le n \le \frac{p-1}{2}}
	\frac{\cos (2\pi nx)}{\sin^2 (2\pi nx)}.
	\]
	Since the summand is invariant under $ n \mapsto p-n $, we have
	\[
	S =
	-\frac{1}{2} \sum_{1 \le n \le p-1}
	\frac{\cos (2\pi nx)}{\sin^2 (2\pi nx)}.
	\]
	Since $ r $ and $ p $ are coprime, the map $ n \mapsto rn $ induces a substitution on $ \Z/p\Z \smallsetminus \{ 0 \} $.
	Thus, we have
	\begin{align}
		S &=
		-\frac{1}{2} \sum_{1 \le n \le p-1}
		\frac{\cos (2\pi n/p)}{\sin^2 (2\pi n/p)}
		=
		-\frac{1}{8} \sum_{1 \le n \le p-1}
		\left(
		\frac{1}{\sin^2 (\pi n/p)} - \frac{1}{\cos^2 (\pi n/p)}
		\right)
		\\
		&=
		-\frac{p^2 - 1}{24}
		+ \frac{1}{8} \sum_{1 \le n \le p-1} \frac{1}{\cos^2 (\pi n/p)}
	\end{align}
	by \zcref{eq:sin_sum}.
	We have
	\begin{align}
		\sum_{1 \le n \le p-1} \frac{1}{\cos^2 (\pi n/p)}
		&=
		\sum_{1 \le n \le p-1} \frac{1}{\sin^2 (\pi/2 - \pi n/p)}
		=
		\sum_{\substack{
				1 \le n \le 2p-1, \\
				\text{$ n $ odd}, \, n \neq p
		}}
		\frac{1}{\sin^2 (\pi n/2p)}
		\\
		&=
		\left(
			\sum_{1 \le n \le 2p-1} 
			- \sum_{\substack{
					1 \le n \le 2p-1, \\
					\text{$ n $ even}
			}}
			- \sum_{n=p}
		\right)
		\frac{1}{\sin^2 (\pi n/2p)}.
	\end{align}
	By \zcref{eq:sin_sum}, we have
	\[
	\sum_{1 \le n \le p-1} \frac{1}{\cos^2 (\pi n/p)}
	=
	\frac{(2p)^2 - 1}{3} - \frac{p^2 - 1}{3} - 1
	=
	p^2 - 1.
	\]
	Thus, we obtain \zcref{eq:sign_TQFT_simple_expression:tirg_sum}.
\end{proof}

We also obtain the following expression.

\begin{cor} \label{cor:sign_TQFT_Eichler_int}
	We have
	\[
	\sigma_2 (x)
	=
	\lim_{\tau \to 0} \left(
		p^2 \frac{16}{\pi^3 \iu} E_{-2}^{\odd} (\tau + x)
		- \frac{8}{\pi \iu} E_0^{\odd} (\tau + x)
		- \frac{1}{3p^2 \tau}
	\right).
	\]
\end{cor}

\begin{proof}
	By \zcref{thm:Dedekind_sum_level_2} \zcref{item:thm:Dedekind_sum_level_2:asymp}, we have
	\begin{align}
		\frac{2^{g+2} g!}{\pi^{g+1} \iu} E_{-g}^{\odd} \left( x + \tau \right)
		&=
		\frac{(-1)^{g/2+1} 2^{g+2}}{g+1}
		\frac{a_{E_{g+2}^{\odd}, 0}^{(x)}}{p^{g+2}} \frac{1}{\tau}
		+ S_g^{\odd} \left( x \right)
		+ O(\tau) \quad \text{ as } \tau \to 0,
		\\
		a_{E_{g+2}^{\odd}, 0}^{(x)} 
		&=
		\frac{2^{g+2} - 1}{4(g+2)} B_{g+2}.
	\end{align}
	Thus, we have
	\begin{alignat}{2}
		\frac{4}{\pi \iu} E_0^{\odd} (\tau + x)
		&=
		-\frac{1}{4p^2 \tau} + S_0^{\odd} \left( x \right)
		& &+ O(\tau) 
		\quad \text{ as } \tau \to 0,
		\\
		\frac{16}{\pi^3 \iu} E_{-2}^{\odd} (\tau + x)
		&=
		-\frac{1}{6p^2 \tau} + S_2^{\odd} \left( x \right)
		& &+ O(\tau) 
		\quad \text{ as } \tau \to 0.
	\end{alignat}
	Therefore, we obtain the claim by \zcref{prop:sign_TQFT_simple_expression_g=2}.
\end{proof}

Finally, we prove \zcref{thm:main}.

\begin{proof}[Proof of \zcref{thm:main}]
	By \zcref{thm:Dedekind_sum_level_2} \zcref{item:thm:Dedekind_sum_level_2:quantum_modular,prop:sign_TQFT_simple_expression_g=2}, we have
	\begin{align}
		\sigma_2 \left( \frac{x}{2x+1} \right) - \sigma_2 (x)
		&=
		p^2 \left( (2x + 1)^2 S_2^{\odd} \left( \frac{x}{2x+1} \right) - S_2^{\odd} (x) \right)
		-2 \left( S_0^{\odd} \left( \frac{x}{2x+1} \right) - S_0^{\odd} (x) \right)
		\\
		&=
		p^2 \left( 2x^2 + 2x + 1 + \frac{1}{p^3 (2r+p)} \right)
		-2 \left( \frac{1}{2} + \frac{1}{2p(2r+p)} \right)
		\\
		&=
		2r^2 + 2rp + p^2 - 1.
	\end{align}
\end{proof}

\section*{Acknowledgements}


The author is sincerely grateful to Julien March\'{e} and Gregor Masbaum for their careful reading of earlier versions of this paper, for many valuable and insightful comments, and for their generous and thoughtful suggestions.
The author would like to thank Toshiki Matsusaka for introducing me to related studies.
The author is supported by JSPS KAKENHI Grant Number JP 20J20308.


\bibliographystyle{alpha}
\bibliography{quantum_inv,modular}

\end{document}